\documentclass{article}
\usepackage{graphicx}
\usepackage{comment}  
\usepackage{float}
\usepackage{amsmath}
\usepackage[margin=1in]{geometry}
\usepackage{tikz}
\usepackage{tikz-cd}

\usepackage{authblk}

\usepackage{amssymb}
\usepackage{amstext}
\usepackage{indentfirst}
\usepackage[utf8]{inputenc}
\usepackage{enumitem}
\usepackage{amsthm}
\usepackage[colorlinks=true, linkcolor=black, citecolor=red]{hyperref}
\usepackage[nameinlink,noabbrev]{cleveref}
\usepackage{tikz}
\usepackage[normalem]{ulem}
\usetikzlibrary{arrows.meta, positioning}
\usetikzlibrary{patterns, shapes.geometric, calc, decorations.pathmorphing}
\usetikzlibrary{calc, fit, shapes.geometric} 

\usetikzlibrary{fit}

\newtheorem{remark}{Remark}
\newtheorem{theorem}{Theorem}[section]
\newtheorem{lemma}[theorem]{Lemma}

\newtheorem{proposition}{Proposition}[section]

\theoremstyle{definition}
\newtheorem{definition}[theorem]{Definition} 

  {\renewcommand\theinnerdefinition{#1}%
   \begin{innerdefinition}}
  {\end{innerdefinition}}

\usepackage{xspace}
\newcommand{\bigoh}[1]{\ensuremath{\mathcal{O}\!\left(#1\right)}\xspace}
\usepackage{algorithm}
\usepackage[noend]{algpseudocode}

\usepackage{amsthm}
\usepackage{cleveref}

\theoremstyle{definition}
\newtheorem{runningex}{Running Example}
\title{Baseball, An Extensive-Form Game-Theoretic Duel}
\date{January, 2026}

\begin{document}


\author[1]{Sebastian E. Ferrando, \thanks{Corresponding author: ferrando@torontomu.ca}}
\affil[1]{Department of Mathematics, Toronto Metropolitan University, Toronto, Canada}

\author[2]{Eli Kohn, \thanks{eli.d.kohn@gmail.com}}
\affil[2]{Department of Mathematics and Statistics, McGill University, Montreal, Canada}

\date{\today} 

\maketitle

\begin{abstract}
We formulate a baseball plate appearance as a finite two-person, zero-sum,
extensive-form game with chance moves and imperfect information. Successive
pitch cycles are connected through the evolving count, while the Batter chooses
an action without observing the Pitcher's current pitch selection. Terminal
outcomes are valued through a run-expectancy utility that incorporates both
runs scored during the plate appearance and the continuation value of the
half-inning. We implement the sequence-form representation of the game, whose
linear-programming formulation grows linearly with the game tree and permits
exact minimax equilibrium computations for trees far beyond the practical
range of the normal form. Nature's transition probabilities are estimated from
MLB Statcast data. We also establish two complementary dynamic-programming
interpretations. For general finite two-person zero-sum games with perfect
recall, we show that the dual variables of the sequence-form best-response
programs decompose into reach weights and conditional continuation values.
Under the additional state-Markov assumptions of the baseball model, we prove
that the full minimax equilibrium can be computed by backward induction over
the twelve non-terminal counts. The resulting computations produce equilibrium
strategies, continuation values, and conditional and reach-weighted measures of
the strategic cost of non-optimal actions.
\end{abstract}
\section*{Introduction}

A baseball plate appearance is naturally a finite-horizon dynamic game. The
strategic interaction unfolds through successive pitch cycles, the count
records the accumulation of earlier outcomes, and each pitch may either
terminate the plate appearance or change the continuation game. The players
also face asymmetric information when choosing their current actions: the
Pitcher selects a pitch, while the Batter chooses whether to Swing or Take
without observing that selection. These features make the plate appearance a
particularly concrete setting in which to study the formulation and exact
computation of equilibrium behavior in a multi-stage game with chance moves,
imperfect information, and endogenous stopping.

Most existing game-theoretic treatments of the Pitcher--Batter confrontation
do not retain this complete dynamic structure. A rich empirical literature has
studied whether professional athletes follow optimal mixed strategies:
Walker and Wooders~\cite{WalkerWooders2001} argue for minimax play in
professional tennis, Kovash and Levitt~\cite{kovash_levitt_2009} find
systematic departures from Nash behavior in baseball and American football,
and Palacios-Huerta~\cite{palacios_huerta_2003} provides supporting evidence
from penalty kicks. These studies, together with most game-theoretic models
of the Pitcher--Batter matchup, treat a play or pitch as an isolated one-shot
matrix game, sometimes specifying a separate game for each count. Sabermetric
work~\cite{tangotiger-re24}, by contrast, develops detailed predictive models
from play-by-play data, most notably through the run-expectancy framework
popularized by Tango, Lichtman, and Dolphin, but does not generally place
those models within a dynamic equilibrium computation.

We model the complete plate appearance as a single extensive-form game in
which the count develops endogenously and determines the available
continuation and stopping possibilities. Within each pitch cycle, the Pitcher
acts first, reflecting the physical order of play, but the Batter's decision
nodes following the possible Pitcher actions belong to a common information
set. The current interaction therefore imposes the same non-observability
restriction as a simultaneous-move game, while the extensive form preserves
the history dependence of the complete plate appearance. To compute its
minimax equilibrium without constructing the exponentially large normal form,
we use the sequence-form representation introduced by von
Stengel~\cite{vonstengel_1996}; see also Koller, Megiddo, and von
Stengel~\cite{koller_megiddo_vonstengel_1996}. Its realization-plan
formulation has size linear in the game tree and permits exact equilibrium
computation for plate appearances involving multiple pitches and richer
action sets. Our objective is not to predict an individual pitch outcome, but
to build and solve a structurally consistent dynamic model in which strategies
may depend on the observed history while respecting the Batter's lack of
information about the current pitch choice.

Our principal computational contribution is not a new sequence-form
algorithm, but the formulation and exact solution of the complete
Pitcher--Batter interaction as an imperfect-information extensive game. As
far as we know, this is the first use of von Stengel's sequence form to
compute an exact equilibrium on the full game tree of a plate appearance.
Following the framework of Koller, Megiddo, and von
Stengel~\cite{koller1994}, sequence-form methods have overcome the
normal-form bottleneck in games with millions of nodes, including
linear-programming approximations for heads-up Texas
Hold'em~\cite{billings2003}, exact equilibrium computation in Rhode Island
Hold'em~\cite{gilpin2005ri}, and iterative double-oracle expansion of large
zero-sum trees~\cite{bosansky2013}. We bring this computational framework to
a sports interaction in which successive stages are linked by an evolving
state, current actions are chosen under asymmetric information, and terminal
outcomes determine both immediate rewards and continuation values. The
sequence form thereby makes it possible to compute equilibrium behavior for
the plate appearance as one connected, multi-stage strategic interaction
rather than as a collection of isolated matrix games.

A particularly close point of comparison is Mercier~\cite{Mercier2024}, who
models the Pitcher--Batter interaction as a family of simultaneous normal-form
games indexed by the baseball state $(R,O,C)$, with $C$ denoting the count.
His payoffs incorporate both immediate and expected future runs, so the two
approaches share a zero-sum formulation, mixed strategic behavior, and a
run-based valuation of outcomes. The principal difference is dynamic. In
Mercier's model, the count selects the pitch-level game to be played; in our
model, the count is generated endogenously within a single extensive plate
appearance and determines the continuation and termination possibilities
available from the current history. Behavioral strategies therefore connect
decisions across the possible evolution of the plate appearance, subject to
the Batter's information, while terminal payoffs depend on the fixed initial
base--out record through the run-expectancy utility. We abstract from foul
balls with two strikes, which would permit arbitrarily long plate appearances
and substantially enlarge the game tree; see Footnote \ref{fn:foul-balls}.

A second baseball paper bears directly on the modeling claim above. Douglas,
Witt, Bendy, and Vorobeychik~\cite{douglas2023atbat} propose a zero-sum
\emph{stochastic game} model of a full baseball at-bat, with the count as the
state variable and on-base percentage as the game's value; they observe that
such a game can in principle be solved by classical stochastic-game methods,
but do not supply a proof that this reduction is exact, do not use the
sequence form, and do not represent within-pitch imperfect information
through an explicit information-set structure. Our count-state recursion in
Section~\ref{subsec:count-state-equilibrium-dp} can be read as making
rigorous, for the run-expectancy utility and under the stated state-Markov
assumptions, the reduction they propose informally; Theorems
\ref{optimalPitcherAtu} and \ref{thm:batter-dp=dual}, by contrast, concern
the general perfect-recall case and have no analogue in their paper.

The paper makes three connected contributions, centered on the exact
computation and interpretation of equilibrium behavior in a dynamic game
with imperfect information. First, we formulate the complete plate appearance
as a single extensive-form game in which the count develops endogenously and
terminal outcomes are valued according to the initial base--out record. We
then implement the sequence-form representation to compute exact minimax
equilibria on the full imperfect-information tree. This replaces optimization
over the exponentially large space of complete pure-strategy plans by linear
programs whose size is linear in the game tree, thereby allowing substantially
richer and longer plate appearances to be treated computationally.

Second, for general finite two-person zero-sum games with perfect recall, we
give a dynamic interpretation of quantities arising from the sequence-form
best-response programs. Using linear-programming duality, strong duality, and
complementary slackness, we show that their dual variables decompose as the
product of the weight of reaching an information set and the conditional value
of continuing from that information set. The resulting quantity coincides
with the \emph{counterfactual value} used in counterfactual regret
minimization~\cite{zinkevich2008cfr}; our contribution is not the quantity
itself, but its direct derivation from von Stengel's sequence-form
best-response program. This also clarifies the distinction between dynamic
programming for a best response against a fixed opponent and computation of
the minimax equilibrium of the full game.

Third, under the additional state-Markov assumptions of the baseball model,
we prove that the full minimax equilibrium admits an exact state reduction and
can be computed by backward induction over the twelve non-terminal counts.
With the full Pitcher action set, the recursion requires solving twelve local
$8\times 2$ zero-sum matrix games. Building on the resulting equilibrium
strategies and continuation values, we introduce conditional and
reach-weighted measures of the cost of a non-optimal action at each count and
record. These measures identify, after accounting for the probability of
reaching each count, the situations in which strategic choices matter most.
Thus, the sequence-form and count-state computations play complementary
roles: the former applies to the full history-dependent game, while the latter
provides a sharper reduction when the count is a sufficient state.

The remainder of the paper is organized as follows. Sections~\ref{sec:game-theory-model}--\ref{sec:plate-appearance-model} formulate the
plate appearance as an extensive-form game with chance moves and imperfect
information and embed it in its baseball context: the game tree, information
partitions, and the Batter's perfect recall in
Section~\ref{sec:game-theory-model}; the base--out records, run-expectancy
utility, and Nature's Statcast-estimated transition kernel in
Section~\ref{sec:plate-appearance-model}. Section~\ref{sec:strategic-form}
shows why the normal form is computationally intractable, and
Section~\ref{sec:sequence-form} introduces the sequence-form realization
plans that avoid this bottleneck. Section~\ref{dynamicProgramming} develops
the two dynamic-programming results described above: the reach-weighted
continuation-value interpretation of the best-response duals
(Theorems~\ref{optimalPitcherAtu} and~\ref{thm:batter-dp=dual}), and the
count-state minimax recursion and its strategic-leverage measures
(Proposition~\ref{prop:count-state-equilibrium-recursion}).
Section~\ref{sec:results} reports numerical equilibria and compares
model-implied swing probabilities to empirical Statcast frequencies. The
appendices collect standard normal-form LP results, the perfect-recall
proof, and the sequence-form programs used in the computations.

\section{Game Theory Model}
\label{sec:game-theory-model}

The model is designed to capture the following features of the interaction:

\begin{enumerate}

\item The plate appearance is treated as a sequential process in which decisions depend on the observed history of play rather than as a collection of independent pitch-level choices. The evolving count and prior actions are therefore incorporated explicitly into the game's records. 

\item The interaction involves imperfect information: the batter does not observe the pitch type prior to the pitch. Consequently, batting decisions are made without knowledge of the pitcher’s realized pitch choice.

\item The extensive-form representation imposes consistency conditions on strategies across histories and information sets. In particular, admissible pitch-selection behavior must arise from behavioral strategies defined on the underlying sequential game.

\item The payoff structure incorporates initial records and continuation values (reflecting how the present plate appearance seats in the dynamical state of the game)  and a stopping rule. These features act as boundary conditions to the full plate appearance. Therefore, strategic actions are evaluated through their effect on the subsequent evolution of the plate appearance and dependent on its embedding into the overall game.

\end{enumerate}

The formal model introduced below is constructed to make these features explicit while remaining computationally tractable through a sequence-form representation.

\subsection{Primitives}

We model a plate appearance in the traditional game theoretic framework as an extensive-form game with chance moves and imperfect information. That is, we consider a rooted infinite directed tree, $\mathcal{T} = (V, E)$, with root $x^0$.
Each vertex $v \in V$ is owned by one of three players 
\[
\mathcal{P} := \{\text{Pitcher (P), Batter (B), Nature (N)}\}
\]
We refer to the Pitcher (P) and the Batter(B) as being {\it strategic players}. Let $V_i \subseteq V$ denote the set of vertices owned by player $i \in \mathcal{P}$.
Then $\{V_P, V_B, V_N\}$ form a partition of the sets of vertices  in $V$, i.e.:
$V  \;=\; V_P \cup V_B \cup V_N \text{ and } V_i \cap V_j \;=\; \varnothing \ \text{for all } i \neq j$.

Let $\mathcal{X}$ be the set of paths in the tree,  i.e. a sequence $\{X_k\}_{k \geq 0}$ of connected nodes with $X_0= x^0$; the framework assumes that a map $\tau: \mathcal{X} \rightarrow \{1, 2, \ldots\}$ is available which is a \emph{stopping time} in the following sense: given  $X, X' \in \mathcal{X}$ such that $X_k=X'_k$,~for ~$0 \leq k \leq \tau(X)$,  it then follows that  $\tau(X')= \tau(X)$. Relying on $\tau$ we will work with the {\it stopped tree} defined as containing finite paths $X=(X_0, X_1, \ldots, X_{\tau(X)})$ and will refer to $X_{\tau(X)}$ as a \emph{terminal node or leaf}  which we will also write as $X_{\tau}$. An explicit stopping time $\tau$  will be introduced in Definition \ref{stoppingTime} below
once some preliminary modeling choices have been made. It should be clear from context when we refer to the background infinite tree or the stopped tree. 

At each node $x$ owned by $i \in \mathcal{P}$, there is a set of available actions:

\begin{itemize}
    \item $A_P = \Theta \times \Lambda$, where $\Theta = \{\text{Fastball}, \text{Offspeed}\}$ and $\Lambda = \{\text{Bottom}, \text{Middle}, \text{Top}, \text{Chase}\}$ is the pitcher's action set.
    \item $A_B = \{\text{Swing}, \text{Take}\}$ is the batter's action set.
    \item $A_N = \{\text{Single}, \text{Double}, \text{Triple}, \text{HomeRun}, \text{Out}, \text{Strike}, \text{Ball}\}$ is Nature's action set.
\end{itemize}

Each directed edge $e = (x, x') \in E$, represents an action taken at node $x$. We write $x \xrightarrow{a} x'$
to denote the edge joining $x$ and  $x'$. The edge is labeled by action $a \in A_i$ for the case that $x \in V_i$; traversing this edge corresponds to player $i$ choosing action $a$, after which the game transitions from $x$ to the successor node $x'$.

Owners of vertices alternate deterministically in a fixed order from Pitcher $\rightarrow$ Batter $\rightarrow$ Nature. The root node $x^0$ is owned by the Pitcher. We refer to Figure \ref{fig:baseball-game-tree} for a visualization of the sequential actions and the game tree.

We also identify nodes with the results of actions, i.e. $V_B \subseteq A_P$, $V_N \subseteq A_B$ and $V_P \subseteq A_N$; it will also follow from our constructions that the leaves for our stopped tree will be actions from nature. On the other hand, the root node $x^0$ will be the empty set (which could be considered to be the first action by nature).

\subsection{The history set $\mathcal{H}$} 
\label{subsec:history_set}
We define the history $h$ associated to a node $x$ as the sequence of action labels encountered along the unique path segment from the root $x^0$ to $x$. Each node $x \in V$ has a unique history since the path segment from the root to $x$ is unique. Formally, a history is a finite sequence of action labels $h=(a_1, \dots, a_k)$ such that there exist vertices $x_1,\dots,x_k\in V$, $x_k=x$, with
\[
x^0 \xrightarrow{a_1} x_1 \xrightarrow{a_2} \cdots \xrightarrow{a_k} x_k.
\]
We denote by $\mathcal{H}$ the set of all such histories, notice that for each such $x$-given history $h=(a_1,\dots,a_k)$, there is a set of paths $\mathcal{X}_x \subseteq \mathcal{X}$ defined by $\mathcal{X}_x =\{X \in \mathcal{X}: X_k=x\}$. 

\begin{definition}[Set $Z$ of terminal histories] \label{terminalHistories}
Let $Z \subset H$ be the set of \emph{terminal histories}, i.e. $h \in Z$ if and only if $h=(a_1, \ldots, a_{\tau(X)})$ where $X_{\tau}=x$ and $x$ is the node associated to $h$.  A terminal history  will represent the entire sequence of actions for a completed plate appearance. The terminal vertex $x\in V$ corresponding to a {\it terminal history} $z \in Z$ has no outgoing edges. 
\end{definition}

We note that, in order to give meaning to the stopping time associated with a given path in the baseball setting, we must retain certain information about the history corresponding to a node in the tree. In particular, we need to keep track of the accumulation of previous actions by nature  that resulted in either a ball or a strike. This is important because, once four balls or three strikes have been accrued during a plate appearance, the plate appearance terminates and results either in a \textit{walk} (with the runner advancing to first) or in an \textit{out}, respectively. We refer to Section \ref{sec:plate-appearance-model} for the introduction of some minimal vocabulary to describe the baseball game.

\begin{definition}[Count]  \label{count}
    Let ${\tt count}: \mathcal{H} \to \mathbb{N}_{\leq 4} \times \mathbb{N}_{\leq 3}$ be a function mapping a history to its pitch count, expressed as the tuple (Balls, Strikes).  For any history $h = (a_1, \dots, a_k)$, we define:
    \[
    {\tt count}(h) = \Biggl(\sum_{i=1}^{k}\mathbf{1}_{a_i = \text{Ball}}, \sum_{i=1}^{k}\mathbf{1}_{a_i = \text{Strike}}\Biggr)
    \]
\end{definition}

One \emph{pitch cycle} consists of the ordered triple of moves
\[
  \bigl(a_{P}, a_{B}, a_{N}\bigr) \in A_{P}\times A_{B}\times A_{N}.
\]
Because a count of four balls or three strikes can be reached only once, and because an in‑play event ends the appearance immediately, every history has a length of at most $6$ pitch cycles.\footnote{This modeling choice is a deviation from the true nature of a plate appearance in baseball. In fact, it is possible to have a foul ball at two strikes, which extends the plate appearance without updating the count. However, as we will see, allowing for this possibility increases the size of the game tree significantly and thus the computational complexity of solving for the minimax solution. Thus, we do not allow for this possibility. \label{fn:foul-balls}}

Next we introduce an explicit stopping time $\tau$ required by our setting, it encodes the end of a plate appearance in our model.
\begin{definition}[Stopping Time $\tau$, Stopped Histories and Terminal Nodes]  \label{stoppingTime}
Consider a path $X \in \mathcal{X}$ we define $\tau: \mathcal{X} \rightarrow \mathbb{N}$ by taking $\tau(X)$ to be the smallest index $k$ such that
$x= X_{k-1}$ is a node in $V$ owned by Nature and there exists a Nature's action  $a_N= a_k$, $X_{k-1} \xrightarrow{a_k} X_{k}$, satisfying:
\[
  \begin{cases}
    a_{N} = \text{Ball}    &\text{and } {\tt count}(h) = (3,s) \text{ with } s \leq 2 \quad\text{(Walk),}\\[4pt]
    a_{N} = \text{Strike}  &\text{and } {\tt count}(h) = (b,2) \text{ with } b \leq 3 \quad\text{(Strikeout),}\\[4pt]
    a_{N} \in \{\text{Single}, \text{Double}, \text{Triple}, \text{HomeRun}, \text{Out}\} &\quad\text{(In-play event),}
  \end{cases}
\]
where $h \in \mathcal{H}$ denotes  the action history associated to $x$. The history $h^\prime = (h, a_N)$, achieved by extending the history $h$ with  Nature's action $a_N$, is then a terminal history (as per Definition \ref{terminalHistories} above). $X_{\tau}= X_k$ is called a \emph{terminal node}.
\end{definition}
From the above definition, it follows that $\tau$ is indeed a stopping time.
Note that $\tau(X) \in \{4, 7, 10, \ldots\}$ as actions along paths come in groups of pitch cycles (each such of length $3$). Therefore the number of pitch cycles can be defined by 
\begin{equation} \label{pitchCycleNumber}
\rho(X) \equiv \frac{\tau(X)-1}{3}.
\end{equation}

\subsection{Information Sets}\label{subsec:information-set}

For each strategic player $i \in \{P, B, N\}$, the set $V_i$ is partitioned into sets $U_i^j, \; j \in \{1, \dots, k_i\}$, where each $U_i^j$ is an \emph{information set}. $\mathcal{U}_i$ is the collection of all information sets $U_i^j, \; j \in \{1, \dots, k_i\}$ and is called the \emph{information partition} of player $i$. We denote $A(U_i^j) := A(x), \;\; x \in U_i^j$, as the action set of the information set. In the specific context of our extensive form game tree, we have the action set at a node $x$ is uniquely determined by the owner of $x$, and all nodes owned by the same strategic player have the same available action set, that is $A(x) = A_i$, for $x \in V_i$. Each information set $U_i^j$ has the property that $A(x) = A(x')$ whenever $x$ and $x'$ are in the same information set $U_i^j$. 

We interpret the nodes in an information set to be indistinguishable. That is, if $x$ and $x'$ are both in $U_i^j$, player $i$ knows that the play of the game reached information set $U_i^j$, but he does not know which specific node in the information set has been reached (\cite{osborne_rubinstein_1994}).

Concretely, we define each information set in the information partition of the Pitcher as a singleton. That is, every information set $U_P^j$ contains exactly one node $x \in V_P$. This corresponds to the Pitcher knowing/observing nature's outputs; in particular, $P$  has \emph{perfect recall}. In the context of the game tree, this means that the Pitcher knows exactly which vertex they are at when they are called on to make a decision. On the other hand, we model the Batter's  uncertainty when facing the present pitch by defining his information partition so that he cannot discern the most recent action taken by the Pitcher. Mathematically, for every Pitcher's node $x \in V_P$, the set of all its immediate child nodes forms a single information set for the batter. Formally, $\mathcal{U}_B = \{U_B^x\}_{x \in V_P}$, where:
\[
U_B^x = \{x' \in V_B \mid (x, x') \in E\}~\mbox{for each}~x \in V_P.
\]

We interpret the above definition of the information partition $\mathcal{U}_B$ as the Batter not being able to discern the most recent action the Pitcher took. Figure \ref{fig:baseball-game-tree} displays $U^1_B$, the first appearing Batter information set.  Importantly, the Batter is able to remember previous actions by the Pitcher and Nature outcomes, as otherwise the Batter would not have the property of perfect recall (see Proposition \ref{prop:batter-perfect-recall} below).

Finally, each information set for Nature is also a singleton corresponding to the possible outputs of the Batter.  Formally, $\mathcal{U}_N = \{U_N^x\}_{x \in V_B}$, where:
\[
U_N^x = \{x' \in V_N \mid (x, x') \in E\}~\mbox{for each}~x \in V_B.
\]

\begin{definition}[Perfect Recall \cite{gametheory_maschler} ] \label{def:perf-recall} 
    A player $i$ is said to have perfect recall if the following two conditions hold: 
    \begin{enumerate}[topsep=0pt, itemsep=6pt, parsep=0pt]
        \item No information set of player $i$ intersects the path from the root to any terminal node more than once.
        \item If two nodes $x, x'$ are in the same information set $U^j_{i}$, then the sequence of player $i$'s past information sets and the actions player $i$ took at those information sets must be identical along the unique paths leading to $x$ and $x'$.
    \end{enumerate}
\end{definition}
Intuitively a player has perfect recall if they remember all actions they made throughout the play of a game.
\begin{proposition}\label{prop:batter-perfect-recall}
The batter has perfect recall.
\end{proposition}
\begin{proof}[Proof that the Batter has perfect recall]
We verify the two conditions in the definition of perfect recall.

First, fix a Batter information set $U_B^x$, where
\[
U_B^x=\{x'\in V_B:(x,x')\in E\}
\]
for some Pitcher node $x\in V_P$. Thus, every node in $U_B^x$ is an
immediate child of $x$. A path from the root to a terminal node can leave
$x$ through only one outgoing edge and therefore can contain at most one
node of $U_B^x$. Hence no Batter information set is encountered more than
once along a terminal path.

Second, let $v,v'\in U_B^x$. Since $v$ and $v'$ are both children of the
same Pitcher node $x$, the unique paths from the root to $v$ and to $v'$
coincide up to and including $x$. They may differ only in the final edge
from $x$, which represents the Pitcher's current, unobserved action.
Consequently, before reaching the current information set $U_B^x$, the Batter has encountered exactly the same earlier information sets and has taken
exactly the same actions at those information sets along both paths.
This verifies the second condition of perfect
recall. Therefore, the Batter has perfect recall.
\end{proof}

Because every information set in the Pitcher's partition $\mathcal{U}_P$ is a singleton, it trivially follows that the pitcher also has perfect recall. \footnote{
This is a standard observation; see \cite{LeytonBrownShoham2008} immediately after Def.~5.2.2 (“Clearly, every perfect-information game is a game of perfect recall”). 
}

\begin{figure}[h!]
\centering
\begin{tikzpicture}[
    scale=0.8, transform shape,
    pitcher/.style={circle, draw=blue!80, fill=blue!10, thick, inner sep=0pt, minimum size=6mm},
    batter/.style={circle, draw=red!80, fill=red!10, thick, inner sep=0pt, minimum size=6mm},
    nature/.style={circle, draw=green!60!black, fill=green!10, thick, inner sep=0pt, minimum size=6mm},
    terminal/.style={rectangle, draw=gray, fill=gray!10, rounded corners, inner sep=3pt, font=\footnotesize},
    edgelab/.style={font=\footnotesize, inner sep=2pt}
]

    \node[pitcher] (P0) at (0,0) {P};
    
    \node[batter] (B1) at (-7,-3.8) {B};
    \node[batter] (B2) at (-5,-3.8) {B};
    \node[batter] (B3) at (-3,-3.8) {B};
    \node[batter] (B4) at (-1,-3.8) {B};
    \node[batter] (B5) at (1,-3.8) {B};
    \node[batter] (B6) at (3,-3.8) {B};
    \node[batter] (B7) at (5,-3.8) {B};
    \node[batter] (B8) at (7,-3.8) {B};
    
    \path[draw, ->, thick] (P0) -- (B1) node[midway, sloped, above, font=\scriptsize] {Fastball, Top};
    \path[draw, ->, thick] (P0) -- (B2) node[midway, sloped, above, font=\scriptsize] {Fastball, Bottom};
    \path[draw, ->, thick] (P0) -- (B3) node[midway, sloped, above, font=\scriptsize] {Changeup, Zone};
    \path[draw, ->, thick] (P0) -- (B4) node[midway, sloped, above, font=\scriptsize] {Changeup, Chase};
    \path[draw, ->, thick] (P0) -- (B5) node[midway, sloped, above, font=\scriptsize] {Slider, Zone};
    \path[draw, ->, thick] (P0) -- (B6) node[midway, sloped, above, font=\scriptsize] {Slider, Chase};
    \path[draw, ->, thick] (P0) -- (B7) node[midway, sloped, above, font=\scriptsize] {Curve, Zone};
    \path[draw, ->, thick] (P0) -- (B8) node[midway, sloped, above, font=\scriptsize] {Curve, Chase};
    
    \draw[dashed, line width=1pt, draw=red!70] (-7.5,-4.2) rectangle (7.5,-3.4);
    \node[color=red!80!black, font=\small\bfseries] at (0,-3.1)
        {Batter's Information Set ($U_B$)};

    \node[nature] (N4S) at (-2,-6) {N};
    \node[nature] (N4T) at (0,-6) {N}; 

    \path[draw, ->, thick] (B4) -- (N4S) node[midway, left, edgelab] {Swing};
    \path[draw, ->, thick] (B4) -- (N4T) node[midway, right, edgelab] {Take};

    \node at (-7,-4.8)  {$\vdots$};
    \node at (-5,-4.8)  {$\vdots$};
    \node at (-3,-4.8)  {$\vdots$};
    \node at (1,-4.8)   {$\vdots$};
    \node at (3,-4.8)   {$\vdots$};
    \node at (5,-4.8)   {$\vdots$};
    \node at (7,-4.8)   {$\vdots$};

    \node[pitcher] (PS)  at (-3,-9.5) {P};
    \node[pitcher] (PD)  at (-2,-9.5) {P};
    \node[pitcher] (PTR) at (-1,-9.5) {P};
    \node[pitcher] (PHR) at (0,-9.5)  {P};
    \node[pitcher] (PO)  at (1,-9.5)  {P};
    \node[pitcher] (PST) at (2,-9.5)  {P};
    \node[pitcher] (PB)  at (3,-9.5)  {P};

    \path[draw, ->] (N4T) -- (PS)  node[midway, sloped, above, font=\scriptsize] {Single};
    \path[draw, ->] (N4T) -- (PD)  node[midway, sloped, above, font=\scriptsize] {Double};
    \path[draw, ->] (N4T) -- (PTR) node[midway, sloped, above, font=\scriptsize] {Triple};
    \path[draw, ->] (N4T) -- (PHR) node[midway, sloped, above, font=\scriptsize] {HR};
    \path[draw, ->] (N4T) -- (PO)  node[midway, sloped, above, font=\scriptsize] {Out};
    \path[draw, ->] (N4T) -- (PST) node[midway, sloped, above, font=\scriptsize] {Strike};
    \path[draw, ->] (N4T) -- (PB)  node[midway, sloped, above, font=\scriptsize] {Ball};

    \node at (-2,-7) {$\vdots$};

    \node[anchor=west, font=\footnotesize, align=left] at (4.2, 1) {
        \textbf{Action Spaces:}\\
        $A_P$ (8 actions): $\Theta \times \Lambda$ \\
        $A_B$ (2 actions): $\{\text{Swing}, \text{Take}\}$\\
        $A_N$ (7 actions): $\{\text{Single}, \text{Double}, \dots\}$
    };

\end{tikzpicture}
\caption{The figure displays the possible actions of our three players and the first information set $U^1_B$ for the Batter.}
\label{fig:baseball-game-tree}
\end{figure}
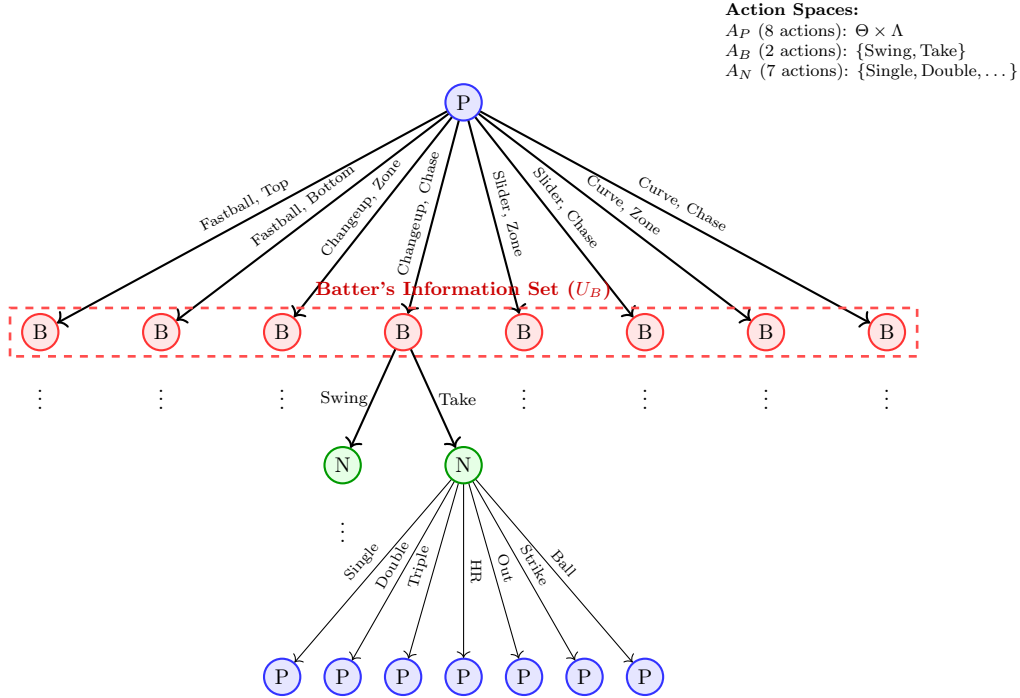

\subsection{Nature's moves and filtered space $(\Omega , \mathcal{F})= \{\mathcal{F}_k\}$.}\label{subsec:nature-moves}

Nature is not a strategic player it selects  actions  by a fixed behavioral strategy $\beta_N$. The later are probability distributions over Nature's action set $A_N$; details on how we model $\beta_N$ are in Section \ref{movesByNature}. For future reference, as we will need to enlist Nature as a formal player in von Stengel's approach, we mention that the information sets $U^j_N$ are singletons and correspond to the children that result from Batter's  actions.

In this section we indicate the factorization of the joint (discrete) probability for each atom $z \in Z$ conditional on a joint strategy for Batter and Pitcher.

It is convenient to introduce the usual base space probabilistic notation   $\Omega := Z$ for the set of terminal histories. Define the set $\mathcal{H}_k := \{h \in \mathcal{H} : |h| \leq k\text{ pitch cycles}\}$ and truncation maps  $\pi_k : \Omega \to \mathcal{H}_k$, for $k \geq 0$, as follows:
\begin{itemize}
    \item The history of $z$ up to and including the $k$'th batter move,
    \item The full history $z$ if the plate appearance ends before $k$ pitch cycles. 
\end{itemize}

The truncation stops after the batter's action — rather than after the full cycle $(a_P^k, a_B^k, a_N^k)$ — so that $\mathcal{F}_k$ captures the information available strictly prior to Nature's move in cycle $k$, ensuring $\mathbb{P}(A_N^k = a \mid \mathcal{F}_k)$ conditions on a $\sigma$-algebra that does not yet contain the outcome being measured.

The maps $\pi_k$ induce an equivalence relation on $\Omega$ where $z \sim_k z'$ if and only if $\pi_k(z) = \pi_k(z')$. For any $z \in \Omega$, its equivalence class is $[z]_k = \{z' \in \Omega : \pi_k(z') = \pi_k(z)\}$. Define the sigma algebra at \textit{time} $k$ as $\mathcal{F}_k := \sigma(\pi_k)$, which is the $\sigma$-algebra generated by the partition of $\Omega$ into these equivalence classes $\{ [z]_k : z \in \Omega \}$. Define random variables $A_i^k: \Omega \to A_i$, $i \in \{P, B, N\}$, as the actions taken by the Pitcher, Batter, and Nature respectively at pitch cycle $k$. The events $\{A_P^k = a_P\}$, $\{A_B^k = a_B\}$, and $\{A_N^k = a_N\}$ are the corresponding measurable subsets of $\Omega$. In particular, the conditional probability of Nature taking action $a \in A_N$ at cycle $k$, given the history up to the Batter's move in that cycle, is:
\[
\mathbb{P}(A_N^k = a \mid \mathcal{F}_k).
\]

Let $z^{\ast} \in \Omega$ be a terminal history associated to terminal node $X^{\ast}_{\tau}$ and $x^{\ast}$ be the node in path $X^{\ast}$ owned by nature during pitch cycle $j$  (naturally we take $j \leq \rho(X^{\ast})$). The associated history $h^{\ast}$, to the node $x^{\ast}$,  corresponds to the sequence of choices made in the first $j$ cycles of the plate appearance up to, and including, the Batter's current action:
\[
h^{\ast}= \pi_j(z) = (a_P^{1*}, a_B^{1*}, a_N^{1*}, a_P^{2*}, a_B^{2*}, a_N^{2*}, \dots, a_P^{j*}, a_B^{j*}).
\]
For each pitch cycle $k \in \{1, \dots, j\}$ we have:
\[
\{A_P^k = a_P^{k*}\}= \{z \in \Omega: A_P^k(z) = a_P^{k*}\} = \{ z \in \Omega : \text{the pitcher's action in cycle } k \text{ is } a_P^{k*} \}
\]
\[
\{A_B^k = a_B^{k*}\}= \{z \in \Omega:A_B^k(z) = a_B^{k*}\} = \{ z \in \Omega : \text{the batter's action in cycle } k \text{ is } a_B^{k*} \}
\]

Next we introduce the (natural) conditional independence assumption for our baseball environment; we assume that the distribution of Nature's choice for $a_N^{k}$ on the current pitch depends exclusively on the actions chosen by the Pitcher and Batter during this specific cycle, and is conditionally independent of all preceding cycles given the current actions. Mathematically, this allows us to collapse the conditioning on the full path  down to the current cycle's events:

\begin{equation}  \label{conditionalIndependenceAssumption}
\mathbb{P}(A_N^k = a_N^{k} \mid \mathcal{F}_k)(z^{\ast}) = \mathbb{P}\!\left(A_N^k = a_N^{k} \mid \{A_P^k = a_P^{k*}\} \cap \{A_B^k = a_B^{k*}\}\right)(z^{\ast}).
\end{equation}
The above expression is conveniently modeled as there is an abundance of data available (see our Section \ref{movesByNature}).  This completes the specification of the probability transition measures at every Nature node across the game tree which is put to work in Section \ref{sec:strategic-form}.

\section{Plate Appearance Model}
\label{sec:plate-appearance-model}

\subsection{Records}  \label{record}
Our baseball model has a \emph{record} quantity that records information of a game at a given plate appearance within a half-inning. We define the record variable $s$ that encodes the runners on base and the number of outs $B$. The variable $s$ will also be referred, occasionally, as a {\it state}. The number of runners {\it on base} as a set is $OB := \{\emptyset, 1, 2, 3,  12, 13, 23, 123\}$. We write the number of {\it outs} in the game as the set $O := \{0, 1, 2\}$. 


\begin{definition} \label{def:state}
    We define the set of records $\mathcal{R}$ as,
\[
    \mathcal{R}:= OB \times O,~~~\mbox{elements of $\mathcal{R}$ are denoted}~s=(ob,o).
    \]    
\end{definition}

For example, $s= (ob, o) =\bigl(12, 2\bigr) \in \mathcal{R}$, corresponds to runners on first base and second base with two outs recorded. 

The following definition is motivated by the need to define a function that maps the current record, and the results of a plate appearance (a terminal history) to a new record.

A function $f$ is introduced next, its intent is to approximate
baseball dynamics in a computationally tractable way. In reality, runner advancement depends on many
contextual variables (e.g., runner speed, batted-ball location, and defensive decisions). Our
implementation abstracts away these factors and instead uses a fixed rule-based heuristic that deterministically updates base occupancy and outs as a function of $z \in Z$.
\begin{definition}[Record-transition map]\label{def:state-dynamics}
We define a function
\[
f:\mathcal{R}\times Z\longrightarrow
\mathcal{R}\cup\{(\emptyset,3)\}
\]
by cases as follows:
\smallskip
\begin{enumerate}
    \item \textbf{Out} If $z \in \{\textsc{Out},\textsc{Strikeout}\}$  and $o < 2$, then one out is recorded ($o'= o+1$) and the base occupancy is unchanged:
    \[
  f((ob, o), z) = (ob, o+1).
    \]
    If $o=2$, the additional out is the third out and ends the current
half-inning. We represent this terminal situation by
\[
f((ob,2),z)=(\emptyset,3).
\]

    \item \textbf{Walk} If $z=\textsc{Walk}$, the batter advances to first and existing runners
    advance only when forced: if first base was occupied then the runner on first is forced to second; if second was also
    occupied then the runner on second is forced to third; if all three bases were occupied then one run scores.
    
    \item \textbf{Single.} If $z=\textsc{Single}$, the batter reaches first and existing runners advance according to a
    deterministic case distinction intended to reflect typical outcomes: runners in scoring position are advanced home,
    and runners on first base advance to third base.
    
    \item \textbf{Double.} If $z=\textsc{Double}$, the batter reaches second; runners on second and third are assumed to score,
    and a runner on first advances to third base.
    \item \textbf{Triple.} If $z=\textsc{Triple}$, all existing runners score and the batter reaches third.
    \item \textbf{Home run.} If $z=\textsc{HR}$, all existing runners score (including the batter) and the bases are cleared.
\end{enumerate}
\end{definition}
The record $(\emptyset,3)$ is used only to represent the end of the
half-inning and is not an element of the ordinary record space
$\mathcal{R}$. For purposes of the utility calculation, we set
\[
RE((\emptyset,3))=0.
\]

\subsection{Utility function}

The basic baseball terminology used in this section is collected in
Appendix~\ref{app:baseball-background}. Fix an initial base--out record
$s^-\in\mathcal{R}$, which places the plate appearance within the larger
evolution of the half-inning. The count and the stopping rule determine the
continuation and termination possibilities available during the plate
appearance; see Definitions~\ref{count} and~\ref{stoppingTime}. A terminal
history $z\in Z$ ends the plate appearance and produces the successor record
\[
s^+=f(s^-,z),
\]
together with $r(s^-,s^+)$ runs scored during the transition. The utility
should therefore reflect both the runs scored immediately and the effect of
the terminal outcome on the expected continuation of the half-inning.

To define the utility of a terminal history $z\in Z$, we use the baseball
measure \emph{Run Expectancy} (RE), popularized by Tom Tango
\cite{tangotiger-re24}. There are eight possible base-occupancy
configurations and three possible numbers of outs, giving the $24$ records in
$\mathcal{R}$. For each record $s\in\mathcal{R}$, $RE(s)$ is the expected
number of runs scored from that record until the end of the half-inning. These
values are estimated empirically by averaging the number of subsequent runs
scored whenever the corresponding record occurs. We use the run-expectancy
values reported by Tom Tango \cite{tangotiger-rematrix}, together with the
terminal convention $RE((\emptyset,3))=0$ introduced after
Definition~\ref{def:state-dynamics}.

\medskip

Importantly, the RE value is sensitive to the record of the game. That is, the utility of a single is larger in a higher leverage record of the game, this matches our intuition that a single with the bases loaded is more valuable than a single with no runners on base. There are other metrics that measure the ``utility" of a terminal event like linear weights. Furthermore, Yee and Deshpande \cite{yee2023} generalized the run expectancy measure by conditioning on more record's parameters like score differential, inning, baserunners and count. For the purposes of our model we are satisfied with the record dependence coarseness of RE. We define the utility of a terminal event  $\{\text{Single, Double, Triple, HomeRun, Out, Walk}\}$ as the difference of the run expectancy of the record before and the record after, plus the runs that were scored between the states.

Writing the {\it records}, before and after the terminal history $h$, as $s^{-}$, and $s^{+}$ respectively (see Definition \ref{def:state} for records), this utility function is formally given by:

\begin{definition}\label{def:utility}
Define the utility functions $u_B, u_P : Z \times \mathcal{R} \to \mathbb{R}$ by:
\[
u(z) \equiv u_B(z, s^-) \equiv RE\bigl(f(s^{-}, z)\bigr) - RE(s^{-}) + r\bigl(s^-, f(s^{-}, z)\bigr)
\]
\[
u_P(z, s^-) := -u_B(z, s^-), \qquad \forall\, z \in Z,\; s^- \in \mathcal{R},
\]
where  $r(s^-, s^+) \in \{0,1,2,3,4\}$ is the {\it number of runs scored during the state transition} $s^- \to s^+$.
Notice that $s^-$ is a fixed external parameter namely, the record at the start of the plate appearance.
\end{definition}
For a fixed initial record $s^-$, the utility of a terminal history depends
only on its terminal plate-appearance outcome and on the resulting record
transition. In particular, an Out and a Strikeout have the same utility,
because they produce the same successor record and score the same number of
runs. Thus, the seven terminal outcomes generate at most six distinct utility
values, and further numerical coincidences may occur for particular
run-expectancy values. The same terminal outcome may also have different
utility under different initial records, since both the initial run expectancy
and the resulting record depend on $s^-$.
\subsection{Modeling Nature's distribution}  \label{movesByNature}

A behavioral strategy for player $i$ is a collection of probability distributions
$\{\beta_i^j\}_{U_i^j\in\mathcal U_i}$, with one distribution associated with each
information set $U_i^j$. For every $U_i^j\in\mathcal U_i$, the distribution
$\beta_i^j$ is defined on the corresponding action set $A(U_i^j)$. For Nature, the behavioral distribution at an information set depends on the
Pitcher and Batter actions in the current pitch cycle. More precisely, if the
Nature information set $U_N^j$ is preceded by the action pair
$(a_P^j,a_B^j)$, we set
\[
\beta_N^j(a_N)
=
\mathbb P\!\left(
A_N=a_N
\,\middle|\,
A_P=a_P^j,\,
A_B=a_B^j
\right),
\qquad a_N\in A_N.
\]
Thus, Nature need not use the same distribution at every information set.
Rather, we assume that the same conditional probability kernel is used at all
Nature nodes preceded by the same current Pitcher--Batter action pair,
independently of the earlier history.

We obtain MLB Statcast pitch-level data for April 2024 from Baseball
Savant~\cite{mlb-statcast}, using the \texttt{baseballr}
package~\cite{baseballr}. Each observed pitch is mapped to the action spaces
of the game. Pitch type is classified as either fastball-family or offspeed,
while pitch location is classified as top, middle, bottom, or chase. The
Batter's action is recorded as either swing or take, and the observed result is
mapped to Nature's action set
\[
A_N
=
\{\text{Ball},\text{Strike},\text{Out},\text{Single},
\text{Double},\text{Triple},\text{HomeRun}\}.
\]

The conditional probabilities defining $\beta_N^j$ are estimated using a
multinomial logistic regression with Nature's action as the categorical
response and the current Pitcher and Batter actions as explanatory variables.
For a fixed baseline outcome $a_N^0\in A_N$, the model takes the form
\[
\log
\frac{
\mathbb P(A_N=a_N\mid A_P=a_P,A_B=a_B)
}{
\mathbb P(A_N=a_N^0\mid A_P=a_P,A_B=a_B)
}
=
\eta_{a_N}(a_P,a_B),
\qquad
a_N\neq a_N^0,
\]
where the functions $\eta_{a_N}$ are determined by the chosen regression
specification.


\section{Strategic Form}
\label{sec:strategic-form}

\subsection{Pure and Mixed Strategies. Game in Normal Form}  \label{sec:strategies}
Suppose we have $L_i$ information sets for player $i \in \{P,B\}$. As in Section \ref{subsec:information-set}, we denote the collection of information sets by $\mathcal{U}_i = \{U^{j}_{i}\}^{L_i}_{j=1}$. Assuming a fixed ordering of the information sets, Definition \ref{pureStrategy} below defines a pure strategy for player $i$ as a tuple of length $|\mathcal{U}_i|$ where the entry at component $j$ is the unique action specified at information set $U_i^j$.

\begin{definition}[Pure strategies \cite{gametheory_maschler}] \label{pureStrategy}
A {\it pure strategy} $s_i$, $i \in \{P, B\}$, is a function from each of the players' information sets to the set of actions available at that the information set. Using the notation $s_i= (s_i^1, \ldots, s_i^{L_i})$, we let:
\[
s^j_i: U^j_i \rightarrow A(U^j_i), \qquad \text{subject to } s^j_i(x)= s^j_i(x') ~\mbox{for all}~~x, x' \in U^j_i
\]
i.e. the $s^j_i$ are constant functions on their domains.
\end{definition}
For example, a strategy for the pitcher can be written as: \[
s_{P}\;=\; (\,a_{P}^{1},\dots, a_{P}^j, \dots ,a_{P}^{L_P}), 
\]
where $s^j_P(U_P^j) = a_P^j$ for all $j \in \{1, \dots, L_P\}$.

Therefore, the set of pure strategies for player $i$ is a cartesian product; i.e. $s_i$ as in Definition \ref{pureStrategy} belongs to:
\begin{equation} \nonumber
 \Sigma_i \equiv \prod_{U^j_i \in \mathcal{U}_i} A(U^j_i).
\end{equation}

\begin{definition}[Strategy Profile]\label{def:strategy-profile}
Let $\Sigma_B$ and $\Sigma_P$ be the set of all pure strategies for batter and pitcher respectively. A {\it strategy profile} is a fixed tuple $(s_B, s_P) \in \Sigma_B \times \Sigma_P$. 
\end{definition}

Note that with a fixed strategy profile, the actions by both the pitcher and batter for the entire plate appearance are predetermined at the start of the game. The variability in the possible terminal outcomes is a result of any unpredictability from nature's moves. The following definition quantifies the conditional probability of a terminal outcome given a fixed strategy profile.

\begin{definition}[Conditional pathwise probabilities] \label{def:reach-probability}
For each tree path $X$, its terminal node $X_{\tau}$ gives an associated terminal history $z \text{ in } Z$
with $z = \bigl(a_{P}^{1}, a_{B}^{1}, a_{N}^{1}, a_{P}^{2}, a_{B}^{2}, a_{N}^{2}, \dots, a_{P}^{k}, a_{B}^{k}, a_{N}^{k} \bigr)$  and  $k = \rho(X)$ (where $\rho(X)$ is given in  (\ref{pitchCycleNumber})).
Let $s_P(U^j_P) =a_P^{j^\star}$ and $s_B(U^j_B)= a_B^{j^\star}$ be a given strategy profile. The (discrete) probability of reaching atom $z$,  conditioned on a strategy tuple $(s_B, s_P)$, is defined to be the probability of the concatenated nature's actions along the terminal history $z$ while being consistent with the fixed actions in $(s_B, s_P)$: 
\begin{equation} \label{probabilityConditionalOnStrategies}
\mu\bigl(z \,\,|\,\,(s_B,s_P)\bigr) = \prod_{j = 1}^k~\mathbb{P}(A_N^j = a_N^{j} \mid \mathcal{F}_j) (z) \cdot {\bf 1}_{\{s_P(z) = a_P^{j^\star}, \, s_B(z) = a_B^{j^\star}\}}(z),
\end{equation}
where $s_i(z) \equiv s_i(X \cap \mathcal{U}_i)= s(X \cap U^j_i)= s(u^j_i)$ and $u^j_i$ is any node in $X \cap U^j_i$. As in \autoref{subsec:nature-moves}, we have the sequence of sigma-algebras $\mathcal{F}_k$ indexed by completed pitch cycles. The expresssion $\mathbb{P}(A_N^j = a_N^{j} \mid \mathcal{F}_j) (z)$ is handled by our assumption in (\ref{conditionalIndependenceAssumption}).

\begin{remark}[Alternative notations in Game Theory]
\label{rem:behavioral_connection}
While Definition \ref{def:reach-probability} deviates somewhat from the behavioral transition specification used in traditional game theory texts the end result is the same.  The standard presentation of extensive-form games, assigns a probability of $1$ to actions aligned with a pure strategy profile and $0$ probability otherwise. Under that paradigm, one would define a joint transition probability $p_i(z \mid s_B, s_P)$ for each pitch cycle $i$ as:
\[
p_j(z \mid s_B, s_P) = \begin{cases} 
\mathbb{P}(A_N^j = a_N^{j^\star} \mid \mathcal{F}_j)(z) & \text{if } s_P(U_P^j) = a_P^{j^\star} \text{ and } s_B(U_B^j) = a_B^{j^\star} \\
0 & \text{otherwise}
\end{cases}
\]
The path probability for the entire sequence would then factorize in terms of these behavioral transitions:
\[
\mu\bigl(z \,\,|\,\,(s_B,s_P)\bigr) = \prod_{i = 1}^k p_i(z \mid s_B, s_P)
\]
In our notation, the indicator function ${\bf 1}_{\{\dots\}}(z)$ acts as this exact piecewise condition: it evaluates to $1$ cycle-by-cycle if the fixed path $z$ is consistent with the strategy profile $(s_B, s_P)$, and instantly zeroes out the entire product if any strategy choice deviates from the terminal history under evaluation.
\end{remark}
\end{definition}

\begin{definition}[Expected Utility] Let $U(s_B,s_P)$ be the expected payoff of the strategy profile $(s_B, s_P)$. That is, 
\begin{equation} \label{expectedPayoff}
U(s_B,s_P) := \sum_{z \in Z}\mu\bigl(z \,\,|\,\,(s_B,s_P)\bigr) \cdot u(z),
\end{equation}
where the utility function $u(z)$ was defined in Definition \ref{def:utility} (in particular $u(z)= u_B(z)$).
\end{definition}

Assuming a zero-sum structure where $U(s_B, s_P)$ represents the expected utility of the batter, the game's payoff matrix $U$ is constructed by computing these values for all $(s_B,s_P) \in \Sigma_B \times \Sigma_P$, where rows correspond to batter strategies and columns correspond to pitcher strategies.

Finally, a {\it mixed strategy} $\pi_i$ for player $i$ is a probability distribution on $\Sigma_i$, the set of those strategies is denoted by $\Delta(\Sigma_i)$. Writing $\pi \equiv (\pi_B, \pi_P)$ allows us to define the expected payoff associated to a mixed strategy: 
\begin{equation} \label{expectedPayoffForMixed}
U(\pi) \equiv \sum_{(s_B, s_P)} \pi_B(s_B)~U(\pi_B,\pi_P)~\pi_P(s_P) = \sum_{(s_B, s_P)}~\sum_{z \in Z}~\pi_B(s_B)~\pi_P(s_P)~\mu\bigl(z \,\,|\,\,(s_B,s_P)\bigr) \cdot u(z),
\end{equation}
A special way to define a mixed strategy is  to randomize pure strategies locally, namely by indicating a probability distribution over the set of possible actions for each information set. Such an alternative way to set probabilities is called a {\it behavioral strategy}.

The specification of $2$-tuples $\pi=(\pi_B, \pi_P)$, where $\pi_i$ are probability distributions on $\Sigma_i$, and $U(\pi)$ as in (\ref{expectedPayoffForMixed}), define what is called a {\it game in normal form}.


\subsection{Combinatorial Complexity and Computational Intractability}
\label{subsec:complexity-count}

The standard normal-form linear programs, ~\cite{gametheory_maschler,agt_nisan},
\eqref{lp:primal-strategic} and \eqref{lp:dual-strategic}; see
Appendix~\ref{linearProgramFormulation}, establish the existence and
characterization of a minimax equilibrium. Their direct computational use,
however, requires the payoff matrix indexed by the players' complete
pure-strategy plans. In a multi-stage extensive game, the number of such plans
grows exponentially with the number of information sets, even though the game
tree itself gives a much more compact description of the interaction. The
following count illustrates the resulting computational obstruction in the
plate-appearance model.

\begin{remark}
The number of pure strategies for the Batter and Pitcher grows exponentially
with the size of the tree. The number of information sets for either player is
$\bigoh{b^m}$, where $b$ is the effective branching factor per pitch cycle and
$m$ is the maximal number of pitch cycles. In our game,
\[
b=8\cdot 2\cdot 2=32,
\]
corresponding to eight Pitcher actions, two Batter actions, and at most two
Nature outcomes, Ball and Strike, that allow the plate appearance to continue.
This is an upper bound: at three balls, a further Ball terminates the plate
appearance; at two strikes, a further Strike terminates it; and at a full count,
neither outcome produces a continuation. As remarked in  Section \ref{subsec:history_set}, the maximum depth of our tree is $6$ pitch cycles. Hence the number of pitcher, batter, and nature information sets is $\bigoh{32^6}$.
By Definition \ref{def:strategy-profile}, the number of pure strategies for the pitcher is $\bigoh{8^{(32^6)}}$, and for the batter $\bigoh{2^{(32^6)}}$.
\end{remark}

Consequently, populating the payoff matrix $U \in \mathbb{R}^{|\Sigma_B| \times |\Sigma_P|}$ requires evaluating $\bigoh{8^{(32^6)} \cdot {2^{(32^6)}}}$ matrix entries. Hence, even though the linear program runs in polynomial time with respect to the input, since the input is exponential with respect to the size of the tree, it is intractable to use \autoref{thm:pure-strat-lp} to solve for a minimax solution. This limitation directly motivates the transition to the sequence-form formulation of von Stengel \cite{koller_megiddo_vonstengel_1996}, which bypasses the pure strategy space entirely by framing the optimization over realizable sequence paths whose growth is strictly linear with respect to the size of the game tree.

The described computational obstruction can be bypassed whenever all players have perfect recall (as is the case in our model). The reason is Kuhn's Theorem  (\cite{kuhn2003lectures}) that indicates that, for each player $i$ with perfect recall, any of his mixed strategies has an {\it equivalent} behavioral strategy. The notion of equivalent here means that the said strategies induce the same terminal probabilities. A behavioral strategy is specified by local action probabilities at the
player's information sets. Thus, the number of probability variables required
to represent such a strategy grows linearly with the size of the game tree,
whereas the number of complete pure strategies underlying a mixed-strategy
representation grows exponentially. This computational saving is incorporated
into von Stengel's sequence-form linear programs, which we describe next.

\section{von Stengel's Sequence Form}\label{sec:sequence-form}
The computational obstruction identified in
Subsection~\ref{subsec:complexity-count} arises from representing strategies
as probability distributions over complete pure-strategy plans. Von Stengel's
sequence form~\cite{vonstengel_1996,vonStengel2007} avoids this obstruction by
representing each strategic player through the sequences of that player's own
actions along the game tree. Pure strategies for the Pitcher and Batter are
therefore replaced by sequences, introduced in
Definition~\ref{definition:sequences}, and behavioral strategies are encoded
by realization plans assigning nonnegative weights to those sequences subject
to linear flow constraints. Nature's moves are incorporated in the same
sequence-based representation, but with the realization plan induced by the
fixed behavioral strategy $\beta_N$ introduced in
Section~\ref{movesByNature}. For games with perfect recall, realization plans
are equivalent to behavioral strategies and lead to linear-programming
formulations whose dimensions grow linearly with the extensive game tree.
This provides the computational representation used below to construct and
solve the Pitcher--Batter game without enumerating its exponentially large
normal form. The presentation in this section follows
\cite{vonstengel_1996}.
\begin{definition} [Sequences] \label{definition:sequences}
The {\it sequence} for a player $i$ defined by a node $n$ of the game tree is the \emph{unique} sequence of edge labels for the edges leaving a node owned by player $i$ on the path from the root to node $n$. Such sequence is denoted $\sigma_i(n)$
\end{definition}
\noindent
Therefore, sequences for a player $i$ are the restriction of a history (see Section \ref{subsec:history_set}) to $i$'s own moves.

\begin{proposition}\label{prop:infoset-sequences}
Every node in an information set $U_P^i$ for Pitcher (respectively Batter) defines the same sequence for the Pitcher (respectively Batter).
\end{proposition}
\begin{proof}
As shown in Section~2 both the pitcher and batter have perfect recall. Thus, by Definition \ref{def:perf-recall} for two nodes $x, \,x' \in U_P^i$, the information sets that intersect the path from the root to $x$, $x'$, as well as the actions taken at those information sets are the same. Thus, it follows that the sequences defined by $x$ and $x'$ must be the same.
\end{proof}

Sequences represent an alternative description of strategies. We inductively define the set of sequences for a strategic player $i \in \{P, B, N\}$ as follows: let $\mathcal{U}_i$ represent the collection of all information sets for player $i$. Consider an information set $v \in \mathcal{U}_i$. By \autoref{prop:infoset-sequences}, every node $x \in v$ shares the identical history of actions chosen by player $i$. We can therefore uniquely define $\sigma_v$ as the sequence of player $i$'s actions leading from the root of the game tree to the information set $v$. If player $i$ takes an action $a \in A_i(v)$ at this information set, we denote the single concatenated sequence that appends action $a$ to the history as $\sigma_v a$. The set of all sequences $S_i$ for player $i$ is then constructed by taking the union of the empty sequence (representing the root) and all possible valid action extensions across their information sets:
\[
S_i := \{\varnothing\} \cup \{ \sigma \equiv \sigma_v a \mid v \in \mathcal{U}_i, \; a \in A_i(v) \},
\]
where $\sigma_v a \equiv \sigma_v \cup \{a\}$ and so $s_i \in S_i$ is a set whose elements are (results of) actions. It follows that the total number of sequences for player $i$ is linear in the size of the game tree: $|S_i| = 1 + \sum_{v \in \mathcal{U}_i} |A_i(v)|$.
This demonstrates the compactness of the sequence form representation when compared to the pure strategy space, whose cardinality grows multiplicatively: $|\Sigma_i| = \prod_{v \in \mathcal{U}_i} |A_i(v)|$.

\subsection{Realization Plans}
Analogous to mixed strategies, a \emph{realization plan} assigns the probability with which a sequence is played. 
Given a behavioral strategy $\beta_i$, the induced realization plan $r_i : S_i \to \mathbb{R}$ is defined on the sequence $s_i \in S_i$ by
\begin{equation}
    r_i(s_i) = \prod_{c \in s_i} ~\beta_i(c).
\end{equation}

We say the realization plan $r_i$, defined above, is induced by $\beta_i$.
\begin{remark}
Below, we will also rely on a more precise notation namely:

\[
\beta_i(a \mid U^j_i),
\]
where $a \in A(U^j_i)$ for a given information set $U^j_i$.
Clearly, for every  $U^j_i \in \mathcal{U}_i$, the functions $\beta_i(\cdot \mid U^j_i)$ satisfy:
\begin{enumerate}
    \item $\beta_i(a \mid U^j_i) \ge 0$ for all $a \in A(U^j_i)$,
    \item $\sum_{a \in A(U^j_i)} \beta_i(a \mid U^j_i) = 1$.
\end{enumerate}
\end{remark}

Instead of using a behavioral strategy to define a realization plan, we can characterize certain structural properties that must be satisfied in order for a function to be a valid realization plan. 
Intuitively, these are network \emph{flow constraints}: a realization plan must assign probabilities to the extended sequences of an information set $v$, such that the sum of the probabilities assigned to the extended sequences $\sigma_v a$ equals the probability assigned to the parent sequence $\sigma_v$. Furthermore, the probability assigned to any sequence must be non-negative. Formally, these linear constraints are written as:

\begin{align}
r_i(\varnothing)
    &= 1,
    \label{realizationPlanConstraints1}\\
r_i(\sigma_v)
    - \sum_{a \in A(v)} r_i(\sigma_v a)
    &= 0,
    \quad \forall\, v \in \mathcal{U}_i,
    \label{realizationPlanConstraints2}\\
r_i(\sigma)
    &\geq 0,
    \quad \forall\,\sigma \in S_i.
    \label{realizationPlanConstraints3}
\end{align}
\begin{definition}[Realization plan] \label{def:realization-plan}
A function $r_i : S_i \to \mathbb{R}$ that satisfies the linear equations (\ref{realizationPlanConstraints1})-(\ref{realizationPlanConstraints3}) is a \emph{realization plan}.
\end{definition}
Notice that constraints (\ref{realizationPlanConstraints1})-(\ref{realizationPlanConstraints3}) are automatically satisfied for Nature as its realization plans are defined by $\beta_N$ which is fixed (being a non-strategic player). $\beta_N$ was introduced in Section \ref{movesByNature}.

Conversely, every realization plan satisfying
Definition~\ref{def:realization-plan} is induced by a behavioral strategy;
see Proposition~\ref{prop:realization-plan-induced} in
Appendix~\ref{vonStengelComputationOfEquilibria}.

\begin{definition}\label{def:seqform-payoff}
Define the payoff function 
\[
g : S_P \times S_B \times S_N \to \mathbb{R},
\]
by:
\[
g(s) := 
\begin{cases}
u(z), & \text{if the sequences}~ s=(s_P, s_B, s_N) \text{is defined by a leaf node}~ z \in Z, \\
0, & \text{otherwise}.
\end{cases}
\]
\end{definition}

We define below the expected payoff given the realization plan for the players. Let $r_N$ be the realization strategy induced by $\beta_N$, where $\beta_N$ is Nature's distribution introduced in Section \ref{movesByNature}. In particular there is a single $r_N$ as there is a single $\beta_N$. 

\begin{definition}[Expected Payoff for a given realization plan]
    Let $r = (r_P, r_B, r_N)$, and $s \in S \equiv S_P \times S_B \times S_N$. Then, the expected payoff for the Batter given $r$ is given by: 
    \begin{equation} \label{expPayoffForGivenRealizationPlan}
    G_B(r) = \sum_{s \in S}g(s) ~ \prod_{i \in \{P, B, N\}}r_i(s_i)
    \end{equation}
where $s=(s_P, s_B, s_N)$. The expected payoff for the Pitcher is computed analogously.
\end{definition}

In games with perfect recall, as is the case of our model, one can replace mixed strategies by behavioral strategies and the latter can be encoded by realization plans. Therefore, the {\it strategic sequence form formulation} given  by realization plans and expected payoff (\ref{expPayoffForGivenRealizationPlan}) represent an equivalent formulation of the original game. See Section 4 of  \cite{vonstengel_1996} for details.

Appendix~\ref{app:realization-behavioral} proves that every valid realization
plan is induced by a behavioral strategy and describes the relations among
pure, mixed, behavioral, and realization-plan representations.
Appendix~\ref{app:sequence-form-linear-programs} gives the sequence-form linear
programs for the Batter's and Pitcher's best-response problems, together with
the coupled programs used to compute a minimax equilibrium. The dual variables
of the best-response programs are interpreted through dynamic programming in
Section~\ref{dynamicProgramming}, while the equilibrium programs provide the
computational basis for the numerical results in Section~\ref{sec:results}.


\section{Dynamic Programming}   \label{dynamicProgramming}

This section develops two complementary connections between dynamic programming
and the sequence-form computation of extensive games. The first applies to the
best-response problem in a general finite two-person zero-sum extensive-form
game with perfect recall. Once the opponent's realization plan is fixed, the
responding player's optimization problem is a linear program whose size is
linear in the game tree and whose dual variables are indexed by that player's
information sets; see
Appendix~\ref{vonStengelComputationOfEquilibria}. We show that these dual
variables have a direct dynamic interpretation. After normalization by the
reach weight generated by the opponent and Nature, they coincide with
conditional continuation values satisfying a Bellman recursion. Equivalently,
the dual variable at an information set $v$ decomposes into the product of the
weight of reaching $v$ and the value of continuing from $v$. This result,
developed for both the Pitcher and Batter and summarized in
Theorems~\ref{optimalPitcherAtu} and~\ref{thm:batter-dp=dual}, connects the
linear-programming computation with statewise dynamic quantities that can be
used to interpret a best response. It applies to general games with perfect
recall, but concerns optimization against a fixed opponent strategy and does
not, in general, replace the coupled sequence-form programs required to
compute a minimax equilibrium.

Subsection~\ref{subsec:count-state-equilibrium-dp} develops a second, more
specialized computational reduction. Under the additional state-Markov
assumptions of the baseball model, the extensive continuation game can be
indexed by the fixed initial base--out record $s^-$ and the current count $c$.
The full minimax equilibrium can then be computed recursively by solving a
local zero-sum matrix game at each non-terminal count. The resulting quantity
$V_{s^-}(c)$ is the minimax continuation value of the plate appearance
conditional on beginning at the state $(s^-,c)$; see
Proposition~\ref{prop:count-state-equilibrium-recursion}. Thus, the first
dynamic program interprets the sequence-form best-response computation in the
general perfect-recall setting, whereas the second computes the equilibrium
itself after exploiting the sufficient-state structure of the baseball model.
Accordingly, $V_{s^-}(c)$ is not a best-response value against a fixed
opponent, is not a sequence-form dual variable, and does not include the
probability of reaching the count. The notation and interpretation of the two
value constructions must therefore remain distinct.

In our developments, and  for convenience, we refer to a running example introduced next.

\begin{runningex} \label{ex:my_example}

Consider the example game tree in Figure~\ref{fig:exact_game_tree}: a mini-reduced plate
appearance where $A_P = \{F \text{ (fastball)}, O \text{ (offspeed)}\}$,
$A_B = \{S \text{ (swing)}, T \text{ (take)}\}$, and
$A_N = \{H \text{ (hit)}, O \text{ (out)}, S \text{ (strike)}\}$. A plate appearance
terminates after two pitches.

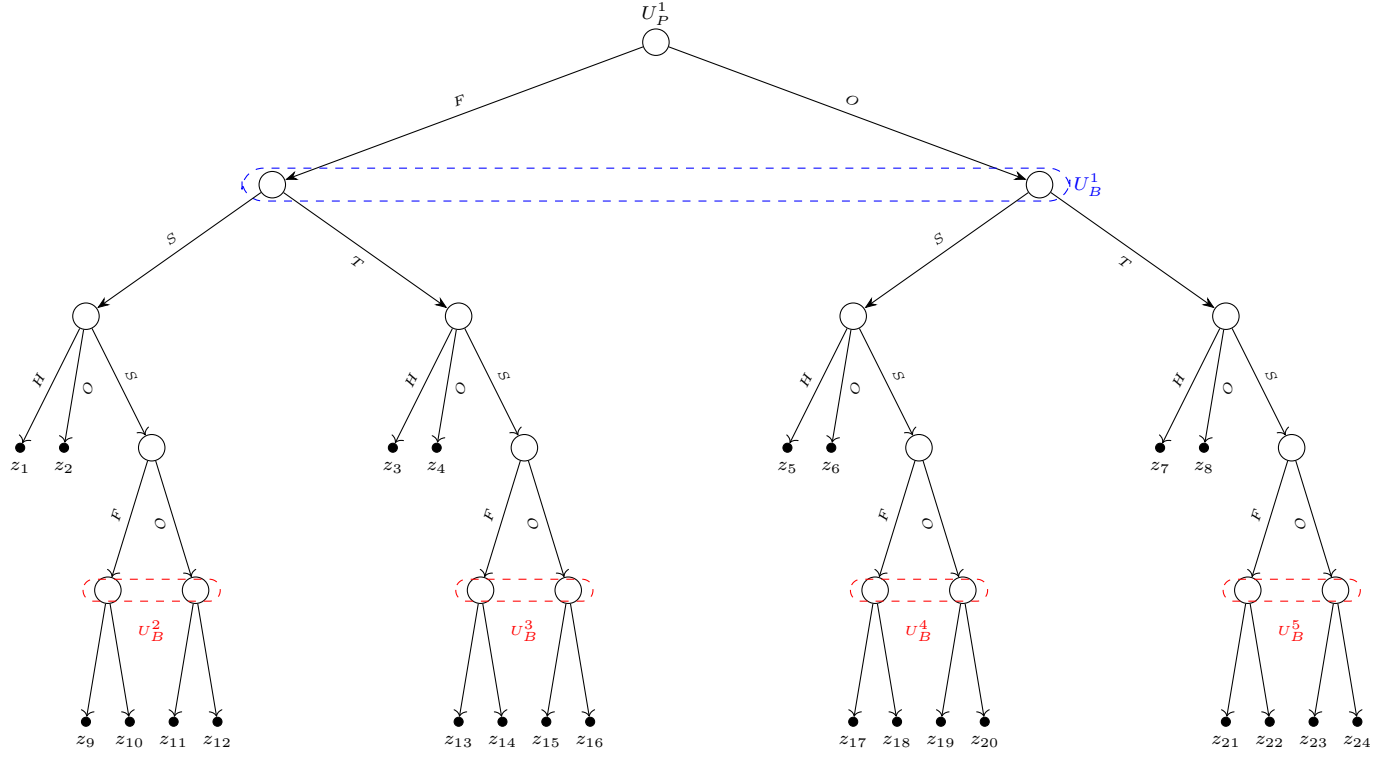
\begin{figure}[h!]
\centering
\begin{tikzpicture}[
    scale=1.45,
    font=\scriptsize,
    dot/.style={circle, draw, fill=black, inner sep=1.2pt},
    circle_node/.style={circle, draw, minimum size=3.5mm, inner sep=0pt},
    edge_label/.style={font=\tiny, midway, above, sloped},
    edge_label_b/.style={font=\tiny, midway, below, sloped}
]
    \node[circle_node] (root) at (0, 6.5) {};
    \node[above=3pt] at (root) {$U^1_P$};

    \node[circle_node] (B1_1) at (-3.5, 5.2) {};
    \node[circle_node] (B1_2) at (3.5, 5.2) {};

    \draw[->, >=Stealth] (root) -- (B1_1) node[edge_label] {$F$};
    \draw[->, >=Stealth] (root) -- (B1_2) node[edge_label] {$O$};

    \draw[dashed, blue, rounded corners=8pt]
        ($(B1_1.west)+(-0.15,0.15)$) rectangle ($(B1_2.east)+(0.15,-0.15)$);
    \node[blue, right=4pt] at (B1_2.east) {$U^1_B$};

    \node[circle_node] (N1) at (-5.2, 4.0) {};
    \node[circle_node] (N2) at (-1.8, 4.0) {};
    \node[circle_node] (N3) at (1.8, 4.0) {};
    \node[circle_node] (N4) at (5.2, 4.0) {};

    \draw[->, >=Stealth] (B1_1) -- (N1) node[edge_label] {$S$};
    \draw[->, >=Stealth] (B1_1) -- (N2) node[edge_label_b] {$T$};
    \draw[->, >=Stealth] (B1_2) -- (N3) node[edge_label] {$S$};
    \draw[->, >=Stealth] (B1_2) -- (N4) node[edge_label_b] {$T$};

    \node[dot] (z1) at (-5.8, 2.8) {}; \node[below=2pt] at (z1) {$z_1$};
    \node[dot] (z2) at (-5.4, 2.8) {}; \node[below=2pt] at (z2) {$z_2$};
    \node[circle_node] (P2) at (-4.6, 2.8) {};
    \draw[->] (N1) -- (z1) node[edge_label] {$H$};
    \draw[->] (N1) -- (z2) node[edge_label_b] {$O$};
    \draw[->] (N1) -- (P2) node[edge_label] {$S$};

    \node[dot] (z3) at (-2.4, 2.8) {}; \node[below=2pt] at (z3) {$z_3$};
    \node[dot] (z4) at (-2.0, 2.8) {}; \node[below=2pt] at (z4) {$z_4$};
    \node[circle_node] (P3) at (-1.2, 2.8) {};
    \draw[->] (N2) -- (z3) node[edge_label] {$H$};
    \draw[->] (N2) -- (z4) node[edge_label_b] {$O$};
    \draw[->] (N2) -- (P3) node[edge_label] {$S$};

    \node[dot] (z5) at (1.2, 2.8) {}; \node[below=2pt] at (z5) {$z_5$};
    \node[dot] (z6) at (1.6, 2.8) {}; \node[below=2pt] at (z6) {$z_6$};
    \node[circle_node] (P4) at (2.4, 2.8) {};
    \draw[->] (N3) -- (z5) node[edge_label] {$H$};
    \draw[->] (N3) -- (z6) node[edge_label_b] {$O$};
    \draw[->] (N3) -- (P4) node[edge_label] {$S$};

    \node[dot] (z7) at (4.6, 2.8) {}; \node[below=2pt] at (z7) {$z_7$};
    \node[dot] (z8) at (5.0, 2.8) {}; \node[below=2pt] at (z8) {$z_8$};
    \node[circle_node] (P5) at (5.8, 2.8) {};
    \draw[->] (N4) -- (z7) node[edge_label] {$H$};
    \draw[->] (N4) -- (z8) node[edge_label_b] {$O$};
    \draw[->] (N4) -- (P5) node[edge_label] {$S$};

    \node[circle_node] (B2_1) at (-5.0, 1.5) {};
    \node[circle_node] (B2_2) at (-4.2, 1.5) {};
    \draw[->] (P2) -- (B2_1) node[edge_label] {$F$};
    \draw[->] (P2) -- (B2_2) node[edge_label_b] {$O$};
    \draw[dashed, red, rounded corners=4pt] ($(B2_1.west)+(-0.1,0.1)$) rectangle ($(B2_2.east)+(0.1,-0.1)$);
    \node[red, font=\tiny, below=8pt] at ($(B2_1)!0.5!(B2_2)$) {$U^2_B$};

    \node[circle_node] (B3_1) at (-1.6, 1.5) {};
    \node[circle_node] (B3_2) at (-0.8, 1.5) {};
    \draw[->] (P3) -- (B3_1) node[edge_label] {$F$};
    \draw[->] (P3) -- (B3_2) node[edge_label_b] {$O$};
    \draw[dashed, red, rounded corners=4pt] ($(B3_1.west)+(-0.1,0.1)$) rectangle ($(B3_2.east)+(0.1,-0.1)$);
    \node[red, font=\tiny, below=8pt] at ($(B3_1)!0.5!(B3_2)$) {$U^3_B$};

    \node[circle_node] (B4_1) at (2.0, 1.5) {};
    \node[circle_node] (B4_2) at (2.8, 1.5) {};
    \draw[->] (P4) -- (B4_1) node[edge_label] {$F$};
    \draw[->] (P4) -- (B4_2) node[edge_label_b] {$O$};
    \draw[dashed, red, rounded corners=4pt] ($(B4_1.west)+(-0.1,0.1)$) rectangle ($(B4_2.east)+(0.1,-0.1)$);
    \node[red, font=\tiny, below=8pt] at ($(B4_1)!0.5!(B4_2)$) {$U^4_B$};

    \node[circle_node] (B5_1) at (5.4, 1.5) {};
    \node[circle_node] (B5_2) at (6.2, 1.5) {};
    \draw[->] (P5) -- (B5_1) node[edge_label] {$F$};
    \draw[->] (P5) -- (B5_2) node[edge_label_b] {$O$};
    \draw[dashed, red, rounded corners=4pt] ($(B5_1.west)+(-0.1,0.1)$) rectangle ($(B5_2.east)+(0.1,-0.1)$);
    \node[red, font=\tiny, below=8pt] at ($(B5_1)!0.5!(B5_2)$) {$U^5_B$};

    \node[dot] (z9)  at (-5.2, 0.3) {}; \node[below=2pt] at (z9)  {$z_9$};
    \node[dot] (z10) at (-4.8, 0.3) {}; \node[below=2pt] at (z10) {$z_{10}$};
    \node[dot] (z11) at (-4.4, 0.3) {}; \node[below=2pt] at (z11) {$z_{11}$};
    \node[dot] (z12) at (-4.0, 0.3) {}; \node[below=2pt] at (z12) {$z_{12}$};
    \draw[->] (B2_1) -- (z9);  \draw[->] (B2_1) -- (z10);
    \draw[->] (B2_2) -- (z11); \draw[->] (B2_2) -- (z12);

    \node[dot] (z13) at (-1.8, 0.3) {}; \node[below=2pt] at (z13) {$z_{13}$};
    \node[dot] (z14) at (-1.4, 0.3) {}; \node[below=2pt] at (z14) {$z_{14}$};
    \node[dot] (z15) at (-1.0, 0.3) {}; \node[below=2pt] at (z15) {$z_{15}$};
    \node[dot] (z16) at (-0.6, 0.3) {}; \node[below=2pt] at (z16) {$z_{16}$};
    \draw[->] (B3_1) -- (z13); \draw[->] (B3_1) -- (z14);
    \draw[->] (B3_2) -- (z15); \draw[->] (B3_2) -- (z16);

    \node[dot] (z17) at (1.8, 0.3) {}; \node[below=2pt] at (z17) {$z_{17}$};
    \node[dot] (z18) at (2.2, 0.3) {}; \node[below=2pt] at (z18) {$z_{18}$};
    \node[dot] (z19) at (2.6, 0.3) {}; \node[below=2pt] at (z19) {$z_{19}$};
    \node[dot] (z20) at (3.0, 0.3) {}; \node[below=2pt] at (z20) {$z_{20}$};
    \draw[->] (B4_1) -- (z17); \draw[->] (B4_1) -- (z18);
    \draw[->] (B4_2) -- (z19); \draw[->] (B4_2) -- (z20);

    \node[dot] (z21) at (5.2, 0.3) {}; \node[below=2pt] at (z21) {$z_{21}$};
    \node[dot] (z22) at (5.6, 0.3) {}; \node[below=2pt] at (z22) {$z_{22}$};
    \node[dot] (z23) at (6.0, 0.3) {}; \node[below=2pt] at (z23) {$z_{23}$};
    \node[dot] (z24) at (6.4, 0.3) {}; \node[below=2pt] at (z24) {$z_{24}$};
    \draw[->] (B5_1) -- (z21); \draw[->] (B5_1) -- (z22);
    \draw[->] (B5_2) -- (z23); \draw[->] (B5_2) -- (z24);
\end{tikzpicture}
\caption{Extensive-form game tree mapping the sequence positions and information sets $U^j$ precisely to outcomes $z_1 \dots z_{24}$.}
\label{fig:exact_game_tree}
\end{figure}

\begin{remark}[Nature Actions and Continuation Paths]
While Nature's full action set is \\
$A_N = \{\text{Single}, \text{Double}, \text{Triple},
\text{HomeRun}, \text{Out}, \text{Strike}, \text{Ball}\}$, the tree at non-terminal counts
only shows two continuation branches out of each Nature node, since any in-play event
(Single, Double, Triple, HomeRun, Out) immediately ends the plate appearance
(Definition~1.2). Only a non-terminal \emph{Ball} or a non-terminal \emph{Strike} allows
play to continue into a subsequent pitch cycle.
\end{remark}

\paragraph{Batter sequence set.}
\[
S_B = \bigl[\, \varnothing,\; S,\; T,\; SS_{2},\; ST_2,\; SS_3,\; ST_3, \; TS_4, \; TT_4, \; TS_5, \; TT_5 \bigr]
\]

\paragraph{Constraint matrix $E$.}
\[
E = \;\;
\begin{array}{c|rrrrrrrrrrr}
 & \stackrel{\scriptsize s_0}{\varnothing}
 & \stackrel{\scriptsize s_1}{S}
 & \stackrel{\scriptsize s_2}{T}
 & \stackrel{\scriptsize s_3}{SS_{2}}
 & \stackrel{\scriptsize s_4}{ST_{2}}
 & \stackrel{\scriptsize s_5}{SS_{3}}
 & \stackrel{\scriptsize s_6}{ST_{3}}
 & \stackrel{\scriptsize s_7}{TS_{4}}
 & \stackrel{\scriptsize s_8}{TT_{4}}
 & \stackrel{\scriptsize s_9}{TS_{5}}
 & \stackrel{\scriptsize s_{10}}{TT_{5}}
\\ \hline
\varnothing & 1 & 0 & 0 & 0 & 0 & 0 & 0 & 0 & 0 & 0 & 0\\
v_{1}       & -1& 1 & 1 & 0 & 0 & 0 & 0 & 0 & 0 & 0 & 0\\
v_{2}       & 0 & -1 & 0 & 1 & 1 & 0 & 0 & 0 & 0 & 0 & 0\\
v_{3}       & 0 & 0 & -1 & 0 & 0 & 1 & 1 & 0 & 0 & 0 & 0\\
v_{4}       & 0 & -1 & 0 & 0 & 0 & 0 & 0 & 1 & 1 & 0 & 0\\
v_{5}       & 0 & 0 & -1 & 0 & 0 & 0 & 0 & 0 & 0 & 1 & 1\\
\end{array}
\]

\paragraph{Pitcher sequence set.}
\[
S_P =
\bigl[
\varnothing,\, F,\, O,\, FF_2,\, FO_2,\, FF_3,\, FO_3,\, OF_4,\, OO_4,\, OF_5,\, OO_5
\bigr]
\]

\noindent Let $F$ be the sequence-form constraint matrix with rows indexed by pitcher
information sets $v_0, v_1, \ldots, v_5$ and columns indexed by $s_0, \ldots, s_{10}$:
\[
F =
\setlength{\tabcolsep}{6pt}
\renewcommand{\arraystretch}{1.15}
\newcommand{\sidx}[1]{\raisebox{.4ex}{\tiny $s_{#1}$}}
\begin{array}{c|*{11}{c}}
 & \sidx{0} & \sidx{1} & \sidx{2} & \sidx{3} & \sidx{4} &
   \sidx{5} & \sidx{6} & \sidx{7} & \sidx{8} & \sidx{9} & \sidx{10}\\[-0.4ex]
 & \varnothing & F & O & FF_{2} & FO_{2} & FF_{3} & FO_{3} & OF_{4} & OO_{4} & OF_{5} & OO_{5}\\
\hline
\varnothing & 1 & 0 & 0 & 0 & 0 & 0 & 0 & 0 & 0 & 0 & 0\\
v_{1}       & -1& 1 & 1 & 0 & 0 & 0 & 0 & 0 & 0 & 0 & 0\\
v_{2}       & 0 & -1& 0 & 1 & 1 & 0 & 0 & 0 & 0 & 0 & 0\\
v_{3}       & 0 & -1& 0 & 0 & 0 & 1 & 1 & 0 & 0 & 0 & 0\\
v_{4}       & 0 & 0 & -1& 0 & 0 & 0 & 0 & 1 & 1 & 0 & 0\\
v_{5}       & 0 & 0 & -1& 0 & 0 & 0 & 0 & 0 & 0 & 1 & 1\\
\end{array}
\]
\end{runningex}

\vspace{.15in}
This section uses $h$ to denote a generic node of the game tree, rather
than the node symbol $x$ of Section~\ref{sec:game-theory-model}, since $x$ (and $y$) have
by now been fixed as the Batter's and Pitcher's realization plans
(Appendix~\ref{vonStengelComputationOfEquilibria}). This notation then overloads $h$, previously the label
for a \emph{history} (Section~\ref{subsec:history_set}), and now also employed for the {\it node that that history defines}.
The two usages of the notation are in a bijection, so no ambiguity results. With this convention, $\sigma_i(h)$ denotes the sequence of Definition~\ref{definition:sequences} -- the
sequence that node $h$ defines for player $i$ -- now made explicit for an arbitrary player
$i \in \{P,B,N\}$ and an arbitrary node $h$, not necessarily one owned by $i$. It
specializes to $\sigma_v$ of Section~\ref{sec:sequence-form} when $h$ belongs to an
information set $v \in \mathcal U_i$ of its own owner $i$, i.e.\ $\sigma_v = \sigma_i(h)$
for any $h \in v$ (well defined by Proposition~\ref{prop:infoset-sequences}); writing
$\sigma_i(h)$ for $i$ other than the owner of $h$ is what lets us speak, below, of the
Batter's or Nature's sequence at a Pitcher node, and vice versa. For $s \in S_P$, $Z[s] \equiv
\{z \in Z : \sigma_P(z) = s\}$ denotes the leaves whose \emph{full} Pitcher sequence equals
$s$ exactly -- i.e., the plate appearance ends at the pitch that produced $s$, with no
further pitcher decision. ($Z[s]$ is defined analogously for $s \in S_B$.)

\textit{A remark on the two players' value functions.} Throughout this section, $v$ (and,
in sums, $v'$) denotes the information set at which a value function or dual variable is
evaluated. On the other hand, $u(z)$, the utility function of Definition~\ref{def:utility}, will always carry an explicit argument $z \in Z$. The Pitcher's and Batter's recursions look
parallel but are \emph{not} literal copies of one another: the Pitcher maximizes his own
payoff $u_P(z) := -u(z)$, while the Batter maximizes $u(z)$ directly, because $u(z)$ is
defined batter-centrically (Definition~\ref{def:utility}) and the sequence-form payoff $g$
of Definition~\ref{def:seqform-payoff} is built from $u(z)$ with no sign attached. This
asymmetry is made precise in Lemma~\ref{lem:matrix-entry} and
Remark~\ref{rem:matrix-entry-batter} below.

\begin{lemma}[Pitcher's payoff matrix in terms of $g$]\label{lem:matrix-entry}
For $s \in S_P$, $(x^\top B)_s = \sum_{z \in Z[s]} \mathbb{P}_x(z)\, u_P(z)$, where $x$ is the
batter's realization plan, $B$ is the pitcher's payoff matrix (Appendix~\ref{vonStengelComputationOfEquilibria}),
$u_P(z) := -u(z)$ is the pitcher's terminal utility (Definition~\ref{def:utility}), and
$\mathbb{P}_x(z) := x(\sigma_B(z))\, r_N(\sigma_N(z))$.
\end{lemma}
\begin{proof}
By Appendix~\ref{vonStengelComputationOfEquilibria}, $B = -A$, so $B_{s_B,s} = -A_{s_B,s} =
-\sum_{s_N \in S_N} r_N(s_N)\, g(s_B,s_N,s)$. Hence
$(x^\top B)_s = \sum_{s_B \in S_B} x(s_B)\, B_{s_B,s} = -\sum_{s_B \in S_B}\sum_{s_N \in S_N}
x(s_B)\, r_N(s_N)\, g(s_B,s_N,s)$. By Definition~\ref{def:seqform-payoff}, $g$ vanishes unless
$(s_B,s_N,s)$ is the triple defined by some leaf $z$, in which case $g = u(z)$, so $-g = u_P(z)$.
Grouping the nonzero terms by the leaf they come from gives the stated sum, since the leaf $z$
with $\sigma_P(z)=s$ is unique once $\sigma_B(z)$ and $\sigma_N(z)$ are also fixed.
\end{proof}

\begin{remark}[The parallel fact for the Batter]\label{rem:matrix-entry-batter}
The same computation, \emph{without} the sign flip -- because $A$ is used directly, not
through the $B=-A$ relation -- gives, for $s \in S_B$,
\[
  (Ay)_s = \sum_{z \in Z[s]} \mathbb{P}_y(z)\, u(z), \qquad \mathbb{P}_y(z) := y(\sigma_P(z))\, r_N(\sigma_N(z)).
\]
This is the plain batter utility $u(z)$ of Definition~\ref{def:utility}, with no sign
change: unlike the Pitcher's problem, the Batter maximizes $u(z)$ itself, since $A$ is her
own payoff matrix. This fact is used implicitly in Section~\ref{subsec:batter-dp} below.
\end{remark}

\subsection{Continuation Value Function and Conditional Probabilities}
\label{subsec:reach-weights}

The quantity $\mathbb{P}_x(z)$ of Lemma~\ref{lem:matrix-entry} is \emph{not} a probability
distribution over $Z$: it deliberately omits the pitcher's own realization weight, because
in the best-response problem the pitcher's sequence is the variable being optimized, not a
random random draw. This subsection makes precise, and proves well defined, the two notions the
Bellman recursion needs: the total weight carried into a Pitcher information set by the
Batter and Nature alone, and the corresponding conditional probabilities.

\begin{definition}[Reach weight]\label{def:reach-weight}
Fix a batter realization plan $x$ and Nature's realization plan $r_N$. For a pitcher
information set $v \in \mathcal U_P$ (a singleton, since the Pitcher has perfect
information, Section~\ref{subsec:information-set}), let $\sigma_B(v) \in S_B$ and $\sigma_N(v) \in S_N$ be the batter and Nature
sequences defined by the unique path from the root to $v$. Define
\[
  \mathbb{P}_x(A_v) \;:=\; x\bigl(\sigma_B(v)\bigr)\cdot r_N\bigl(\sigma_N(v)\bigr).
\]
\end{definition}

\noindent The event $A_v$ is simply $\{z \in Z : v \preceq z\}$, $\preceq$ is the natural order along paths of the game tree and so  $A_v$ are the leaves reachable
through $v$.  Lemma~\ref{lem:conservation} below shows $\mathbb{P}_x(A_v)$ is exactly the
total $\mathbb{P}_x$-weight of $A_v$, so the notation is consistent with reading
$\mathbb{P}_x(A_v)$ as "the probability, under $x$ and $r_N$ alone, of reaching $v$."

\begin{lemma}[Conservation]\label{lem:conservation}
For every $v \in \mathcal U_P$ and every $a \in A(v)$,
\begin{equation}\label{eq:conservation}
  \mathbb{P}_x(A_v) \;=\; \sum_{z \in Z[\sigma_v a]} \mathbb{P}_x(z) \;+\; \sum_{v' \in \mathcal U_P:\, \sigma_{v'} = \sigma_v a} \mathbb{P}_x(A_{v'}).
\end{equation}
In particular (taking $v$ to be the root and telescoping), $\mathbb{P}_x(A_v) = \sum_{z \in A_v} \mathbb{P}_x(z)$.
\end{lemma}
\begin{proof}
Fix $a \in A(v)$. Every continuation of play after the pitcher plays $a$ at $v$ passes
through exactly one batter information set and then one Nature move, after which either a
leaf is reached (contributing to the first sum) or a new pitcher information set $v'$ with
$\sigma_{v'} = \sigma_v a$ is reached (contributing to the second). Both the batter's and
Nature's realization plans satisfy the flow constraint~\eqref{realizationPlanConstraints2}:
the weight $r_i(\sigma)$ assigned to a sequence equals the sum of the weights assigned to
all one-step extensions of $\sigma$. Applying this constraint repeatedly, layer by layer,
from $\sigma_B(v)\cdot r_N(\text{node }v)$ down to the leaves and next-pitcher-nodes below
$(v,a)$ telescopes exactly to the right-hand side of~\eqref{eq:conservation}; the left-hand
side does not depend on $a$ because $a$ is a pitcher action and affects neither
$\sigma_B(v)$ nor $\sigma_N(v)$. The displayed special case follows by taking $v$ to be the
root, where $\mathbb{P}_x(A_{\mathrm{root}}) = x(\varnothing) r_N(\varnothing) = 1$, and
unfolding~\eqref{eq:conservation} recursively over all descendants.
\end{proof}

For $v$ with $\mathbb{P}_x(A_v) > 0$ we define conditional probabilities by simple
normalization,
\[
  \mathbb{P}_x(z \mid v) := \frac{\mathbb{P}_x(z)}{\mathbb{P}_x(A_v)}\ \ (z \in A_v), \qquad
  \mathbb{P}_x(A_{v'} \mid v) := \frac{\mathbb{P}_x(A_{v'})}{\mathbb{P}_x(A_v)}\ \ (v' \succeq v).
\]
No separate notion of "conditioning on the action $a$" is needed: once $v$ and $a$ are
fixed, $Z[\sigma_v a]$ and $\{v' : \sigma_{v'} = \sigma_v a\}$ are simply the (disjoint) subsets
of $A_v$ compatible with that choice, and~\eqref{eq:conservation} already accounts for all
of $\mathbb{P}_x(A_v)$ being split between them. It is also convenient to condition on the
smaller event $Z[\sigma_va] \subseteq A_v$ itself, writing, for $\mathbb{P}_x(\sigma_va\mid v)
:= \sum_{z\in Z[\sigma_va]}\mathbb{P}_x(z)/\mathbb{P}_x(A_v) > 0$,
\[
  \mathbb{P}_x(z \mid \sigma_v a) := \frac{\mathbb{P}_x(z)}{\sum_{z'\in Z[\sigma_va]} \mathbb{P}_x(z')}\ \ (z \in Z[\sigma_va]);
\]
this is the same normalization used below in $\bar g_P(v,a)$, restricted to the leaves
reachable through the specific extension $\sigma_va$ rather than all of $A_v$.

\begin{definition}[Pitcher's normalized value function]\label{def:pitcher-V}
Define $V : \mathcal U_P \to \mathbb{R}$, for every $v$ with $\mathbb{P}_x(A_v) > 0$, by the
backward recursion (well founded, since the tree is finite,
Section~\ref{subsec:history_set})
\begin{equation}\label{eq:bellman-pitcher}
  V(v) \;=\; \max_{a \in A(v)}\ \Bigl\{\, \mathbb{P}_x(\sigma_v a \mid v)\, \bar g_P(v,a) \;+\; \sum_{v':\,\sigma_{v'}=\sigma_va} \mathbb{P}_x(A_{v'} \mid v)\, V(v') \Bigr\},
\end{equation}
where $\bar g_P(v,a) := \sum_{z\in Z[\sigma_va]} \mathbb{P}_x(z\mid \sigma_v a)\, u_P(z)$ is the
conditional expected \emph{pitcher} payoff of ending the plate appearance on this pitch --
equivalently, minus the conditional expectation, given the extension $\sigma_v a$, of the
sequence-form payoff $g$ of Definition~\ref{def:seqform-payoff}, which returns the batter's
terminal utility $u(z)$ of Definition~\ref{def:utility} at any leaf $z$; recall
$u_P(z):=-u(z)$. The sign is forced by Lemma~\ref{lem:matrix-entry}: $V$ is defined so that
$V(v)\cdot\mathbb{P}_x(A_v)$ matches the dual variable $q^\star(v)$ of the pitcher's
\emph{own} best-response LP (Theorem~\ref{optimalPitcherAtu} below), and that LP is stated
in terms of $B=-A$, not $A$. Also
$\mathbb{P}_x(\sigma_v a \mid v) := \sum_{z \in Z[\sigma_v a]} \mathbb{P}_x(z) / \mathbb{P}_x(A_v)$.
(Information sets unreachable under $x$, i.e.\ $\mathbb{P}_x(A_v)=0$, play no role in any
best response and are assigned an arbitrary value, as is standard. We subscript $\bar g_P$
by the player, and use $g$ rather than $u$ or $v$ in its name, so that three distinct
objects -- the batter's terminal utility $u(z)$, this section's information-set variable
$v$, and the player-specific conditional payoff being defined here -- never share a symbol.
Section~\ref{subsec:batter-dp} defines the parallel $\bar g_B$, built from the plain
$u(z)$ rather than $u_P(z)$, since the Batter needs no sign flip -- Remark~\ref{rem:matrix-entry-batter}.)
\end{definition}

\subsection{Bellman Equation and the Pitcher's Best Response}

Recall from Appendix~\ref{vonStengelComputationOfEquilibria} that a pitcher best response
to $x$ solves the primal-dual pair
\[
\max_{y \ge 0}\; (x^\top B) y \ \ \text{s.t.}\ \ Fy = f,
\qquad\qquad
\min_{q}\; f^\top q \ \ \text{s.t.}\ \ F^\top q \ge B^\top x,
\]
whose dual constraints, written information-set by information-set, read
\begin{equation}
\label{eq:p-dual-ineq}
q(v) \;\ge\; (x^\top B)_{\sigma_v a} \;+\; \sum_{v' \in \mathcal U_P :\, \sigma_{v'} = \sigma_v a} q(v'), \qquad \forall\, v \in \mathcal U_P,\ a \in A(v).
\end{equation}

\begin{theorem}\label{optimalPitcherAtu}
Let $q^\star$ be an optimal dual solution to the pitcher's best-response LP against a
fixed batter realization plan $x$. Then for every $v \in \mathcal U_P$ with $\mathbb{P}_x(A_v)>0$,
\[
  q^\star(v) \;=\; V(v)\cdot \mathbb{P}_x(A_v).
\]
\end{theorem}
\begin{proof}
By backward (structural) induction on the tree, from the leaves of $\mathcal U_P$ upward.

\smallskip\noindent\textbf{Base case: $v$ has no pitcher information set below it}, i.e.\
$\{v' : \sigma_{v'} = \sigma_v a\} = \varnothing$ for every $a \in A(v)$. Inequality~\eqref{eq:p-dual-ineq}
reads $q(v) \ge (x^\top B)_{\sigma_v a}$ for all $a\in A(v)$, so by strong LP duality
$q^\star(v) = \max_{a\in A(v)} (x^\top B)_{\sigma_va}$. Dividing by $\mathbb{P}_x(A_v) > 0$
(a constant that does not depend on $a$, Definition~\ref{def:reach-weight}) and using
Lemma~\ref{lem:matrix-entry},
\[
\frac{q^\star(v)}{\mathbb{P}_x(A_v)} = \max_{a\in A(v)} \frac{\sum_{z\in Z[\sigma_va]}\mathbb{P}_x(z)\,u_P(z)}{\mathbb{P}_x(A_v)}
= \max_{a\in A(v)} \mathbb{P}_x(\sigma_v a\mid v)\,\bar g_P(v,a),
\]
which is exactly~\eqref{eq:bellman-pitcher} with the second (continuation) term absent, so
$V(v) = q^\star(v)/\mathbb{P}_x(A_v)$.

\smallskip\noindent\textbf{Inductive step.} Suppose the claim holds for every pitcher
information set $v'$ with $\sigma_{v'} = \sigma_v a$ for some $a \in A(v)$ (all children of $v$
in the pitcher's decision tree). Dividing~\eqref{eq:p-dual-ineq} by $\mathbb{P}_x(A_v)$ and
substituting the induction hypothesis $q(v') = V(v')\,\mathbb{P}_x(A_{v'})$,
\[
\frac{q(v)}{\mathbb{P}_x(A_v)} \;\ge\; \frac{(x^\top B)_{\sigma_va}}{\mathbb{P}_x(A_v)} \;+\; \sum_{v':\,\sigma_{v'}=\sigma_va} \frac{\mathbb{P}_x(A_{v'})}{\mathbb{P}_x(A_v)}\,V(v'), \qquad \forall a\in A(v),
\]
i.e.\ $q(v)/\mathbb{P}_x(A_v) \ge$ the right-hand side of~\eqref{eq:bellman-pitcher} for every
$a$, hence $q(v)/\mathbb{P}_x(A_v) \ge V(v)$. For the reverse inequality, complementary
slackness gives an action $a^\star$ with $y^\star(\sigma_v a^\star) > 0$ at which
\eqref{eq:p-dual-ineq} holds with equality; running the same division and substitution at
$a^\star$ shows $q^\star(v)/\mathbb{P}_x(A_v)$ equals the $a^\star$ term of the maximum
in~\eqref{eq:bellman-pitcher}, which is $\le V(v)$. Hence $q^\star(v) = V(v)\,\mathbb{P}_x(A_v)$.
\end{proof}

\begin{remark}
The proof uses only Lemma~\ref{lem:conservation} and complementary slackness; there is no
step of the form "$\mathbb{P}(A_a)=1$ at the argmax," which was never a well-defined
statement. The pitcher's own action $a$ is not a random variable in this computation --
it is the decision variable of the very best-response problem being solved.
\end{remark}
\begin{remark}[Relation to counterfactual value]
The quantity $V(v)\,\mathbb{P}_x(A_v)$ in Theorem~\ref{optimalPitcherAtu} is
the same object as the \emph{counterfactual value} of information set $v$
against the fixed strategy pair $(x,r_N)$, as used in counterfactual regret
minimization~\cite{zinkevich2008cfr}: the expected payoff at $v$ weighted by
the reach probability of everyone other than the acting player. That
quantity is normally defined directly from a strategy profile and shown to
decompose recursively; Theorem~\ref{optimalPitcherAtu} instead derives it as
the dual variable of von Stengel's sequence-form best-response program, via
strong duality and complementary slackness. The two routes agree on the
value; the derivation given here is, as far as we know, new.
\end{remark}

\begin{remark}[Relation to von Stengel's original treatment]
That a sequence-form best response can be computed by backward induction is already
implicit in von Stengel~\cite{vonstengel_1996}, Section 6, stated there generically for
whichever player's best response is being computed. What Theorem~\ref{optimalPitcherAtu}
adds to that account is the explicit split $q^\star(v) = V(v)\cdot \mathbb{P}_x(A_v)$ of
the (unnormalized) dual variable into a reach probability and a per-unit-probability
continuation value, together with the proof, via Lemma~\ref{lem:conservation} and
complementary slackness, that this split is exact.
\end{remark}

\begin{runningex}[Continued] \label{ex:my_example-cont}
$v_1$
below denotes the root pitcher information set, drawn as $U^1_P$ in
Figure~\ref{fig:exact_game_tree}, and $v_2,\ldots,v_5$ the four Pitcher information sets one
level down -- each a singleton, since the Pitcher has perfect information, reached after one round of batter and Nature
moves.
With $A(v_1)=\{F,O\}$, Lemma~\ref{lem:conservation}
gives $\mathbb{P}_x(A_{v_1}) = \sum_{z\in Z[F]}\mathbb{P}_x(z) + \mathbb{P}_x(A_{v_2}) + \mathbb{P}_x(A_{v_3})$
for the choice $a=F$, and $\mathbb{P}_x(A_{v_1}) = \sum_{z\in Z[O]}\mathbb{P}_x(z) + \mathbb{P}_x(A_{v_4}) + \mathbb{P}_x(A_{v_5})$
for the choice $a=O$. Using Lemma~\ref{lem:matrix-entry},
$\sum_{z\in Z[F]}\mathbb{P}_x(z)u_P(z) = (x^\top B)_{s_1}
= x(S)\,r_N(H)\,u_P(z_1) + x(S)\,r_N(O)\,u_P(z_2) + x(T)\,r_N(H)\,u_P(z_3) + x(T)\,r_N(O)\,u_P(z_4)$,
matching the constraint-matrix row for $s_1=F$ in $F$ above (recall $u_P=-u$), and likewise
$\sum_{z\in Z[O]}\mathbb{P}_x(z)u_P(z) = (x^\top B)_{s_2}
= x(S)\,r_N(H)\,u_P(z_5) + x(S)\,r_N(O)\,u_P(z_6) + x(T)\,r_N(H)\,u_P(z_7) + x(T)\,r_N(O)\,u_P(z_8)$,
matching the row for $s_2=O$. Equation~\eqref{eq:bellman-pitcher} then reads
\begin{equation} \nonumber
V(v_1) = \max\Bigl\{\, \mathbb{P}_x(F\mid v_1)\,\bar g_P(v_1,F) + \mathbb{P}_x(A_{v_2}\mid v_1)V(v_2) + \mathbb{P}_x(A_{v_3}\mid v_1)V(v_3),
\end{equation}
\begin{equation} \nonumber
\mathbb{P}_x(O\mid v_1)\,\bar g_P(v_1,O) + \mathbb{P}_x(A_{v_4}\mid v_1)V(v_4) + \mathbb{P}_x(A_{v_5}\mid v_1)V(v_5) \Bigr\}.
\end{equation}
\end{runningex}

\subsection{Batter Best Response Dynamic Programming}
\label{subsec:batter-dp}

The batter's problem is structurally identical, with one difference: because the batter
does not observe the pitch type, a batter information set $v \in \mathcal U_B$ may contain
several nodes $h \in v$ (one for each pitcher sequence consistent with what the batter has
observed so far). The reach weight must therefore aggregate over all such $h$.

\begin{definition}[Batter reach weight]\label{def:batter-reach-weight}
Fix a pitcher realization plan $y$ and Nature's plan $r_N$. For a node $h$, let
$\sigma_P(h)$, $\sigma_N(h)$ be the pitcher and Nature sequences on the path to $h$, and set
$w_y(h) := y(\sigma_P(h))\,r_N(\sigma_N(h))$. For a batter information set $v \in \mathcal U_B$,
define
\[
  \mathbb{P}_y(A_v) \;:=\; \sum_{h \in v} w_y(h).
\]
\end{definition}

\begin{lemma}\label{lem:conservation-batter}
For every $v \in \mathcal U_B$ and $a \in A(v)$,
\[
  \mathbb{P}_y(A_v) = \sum_{z \in Z[\sigma_v a]} \mathbb{P}_y(z) + \sum_{v' \in \mathcal U_B:\,\sigma_{v'}=\sigma_va} \mathbb{P}_y(A_{v'}),
\]
where $\mathbb{P}_y(z) := y(\sigma_P(z))\,r_N(\sigma_N(z))$. In particular $\mathbb{P}_y(A_v) = \sum_{z \in \bigcup_{h\in v} A_h}\mathbb{P}_y(z)$.
\end{lemma}
\begin{proof}
Apply Lemma~\ref{lem:conservation}'s argument separately to each $h \in v$ (using the
pitcher's and Nature's flow constraints below $h$) and sum over $h \in v$; the batter's own
action $a$ does not affect $w_y(h)$ for $h \in v$, so, exactly as in the pitcher's case, the
left-hand side is common to all $a \in A(v)$.
\end{proof}

Define $V(v) := p(v)/\mathbb{P}_y(A_v)$, by the backward recursion
\begin{equation}\label{eq:bellman-batter}
  V(v) = \max_{a\in A(v)} \Bigl\{\, \mathbb{P}_y(\sigma_va\mid v)\,\bar g_B(v,a) + \sum_{v':\,\sigma_{v'}=\sigma_va} \mathbb{P}_y(A_{v'}\mid v)\,V(v') \Bigr\},
\end{equation}
where $\bar g_B(v,a) := \sum_{z\in Z[\sigma_va]} \mathbb{P}_y(z\mid \sigma_v a)\, u(z)$ is the
conditional expected \emph{batter} payoff of ending the plate appearance on this pitch, and
$\mathbb{P}_y(\sigma_va\mid v)$ is defined exactly as $\mathbb{P}_x(\sigma_va\mid v)$ in
Definition~\ref{def:pitcher-V}, with $x,\mathbb{P}_x$ replaced by $y,\mathbb{P}_y$.
Unlike the pitcher's $\bar g_P$, no sign flip is needed: $\bar g_B$ is built from the plain
utility $u(z)$, not $u_P(z)$, since $A$ is the batter's own payoff matrix
(Remark~\ref{rem:matrix-entry-batter}) and the batter's LP below is a direct maximization
of $x^\top(Ay)$, not of a negated matrix. Recall from Appendix~\ref{vonStengelComputationOfEquilibria}
that the batter's best-response dual constraints are
\begin{equation}
\label{eq:batter-lp}
p(v) \;\ge\; (Ay)_{\sigma_va} \;+\; \sum_{v': \sigma_{v'} = \sigma_va} p(v'), \qquad \forall v \in \mathcal U_B,\ a\in A(v).
\end{equation}

\begin{theorem}[Batter BR DP equals LP dual]\label{thm:batter-dp=dual}
Let $p^\star$ be an optimal dual solution of the batter's best-response LP against a fixed
pitcher realization plan $y$. Then for every $v \in \mathcal U_B$ with $\mathbb{P}_y(A_v)>0$,
\[
  p^\star(v) \;=\; V(v)\cdot \mathbb{P}_y(A_v).
\]
\end{theorem}
\begin{proof}
Identical to the proof of Theorem~\ref{optimalPitcherAtu}, with Lemma~\ref{lem:conservation}
replaced by Lemma~\ref{lem:conservation-batter}, Remark~\ref{rem:matrix-entry-batter} playing
the role Lemma~\ref{lem:matrix-entry} played there (unsigned, since $A$ needs no sign flip),
and~\eqref{eq:p-dual-ineq} replaced
by~\eqref{eq:batter-lp}: the base case divides the equality
$p^\star(v)=\max_{a}(Ay)_{\sigma_va}$ (strong duality, no further batter decisions below
$v$) by $\mathbb{P}_y(A_v)$; the inductive step divides~\eqref{eq:batter-lp} by
$\mathbb{P}_y(A_v)$, substitutes the induction hypothesis on each child $v'$, and invokes
complementary slackness at the batter's optimal action to obtain equality. Aggregating over
$h \in v$ changes none of the algebra, since $\mathbb{P}_y(A_v)$ is a single constant not
depending on $a$ (Lemma~\ref{lem:conservation-batter}).
\end{proof}


%

\subsection{A Count-State Minimax Recursion and Strategic Leverage}
\label{subsec:count-state-equilibrium-dp}

The Bellman recursions developed above interpret the computation of a best
response after the opponent's realization plan has been fixed. In a general
two-person zero-sum extensive-form game with imperfect information, such a
single-player dynamic program does not replace the coupled sequence-form
linear programs required to compute a minimax equilibrium. The present
baseball model, however, has additional state-Markov structure. Under the
assumptions stated below, all histories sharing the same initial base--out
record and current count induce the same continuation game. This permits an
exact compression of the extensive game and a second dynamic program that
computes the minimax equilibrium itself; see
Proposition~\ref{prop:count-state-equilibrium-recursion}. The resulting
recursion requires only a sequence of local zero-sum matrix-game computations
and also produces the action values used to define conditional and
reach-weighted measures of strategic leverage.

\subsubsection{The count-state reduction}

Fix an initial base--out record $s^-\in\mathcal R$ and write
\[
  \mathcal C:=\{0,1,2,3\}\times\{0,1,2\}
\]
for the twelve non-terminal possible {\tt count} function values.  
The record $s^-$ is fixed
throughout the plate appearance and enters the terminal utility through
$u(\,\cdot\,;s^-)$.  The {\tt count} values $c=(b,s)\in\mathcal C$ (as per Definition \ref{count}) are the component of
the state of the game that evolve after each non-terminal Ball or Strike.  Thus, for the
present model, the relevant game context is the pair $(s^-,c)$: one solves the
count recursion separately for each initial base--out record under study.

We formulate the reduction slightly more generally than for our current model.  Let
\[
  \beta_N(n\mid s^-,c,a_P,a_B), \qquad n\in A_N,
\]
be Nature's conditional distribution.  We require only that it depends on the
past through the current public state $(s^-,c)$ and the current actions
$(a_P,a_B)$.  Our model for Nature's behavioral distribution,  see Section~\ref{movesByNature}, which
depends only on $(a_P,a_B)$, is a special case.  

For $n\in A_N$, write $\nu(n)$ for the plate-appearance outcome it names
when it ends the plate appearance: $\nu(n):=n$ for
$n\in\{\mathrm{Single},\mathrm{Double},\mathrm{Triple},\mathrm{HomeRun},\mathrm{Out}\}$,
$\nu(\mathrm{Ball}):=\mathrm{Walk}$, and $\nu(\mathrm{Strike}):=\mathrm{Strikeout}$.

For $c=(b,s)\in\mathcal C$, say $n\in A_N$ is \emph{terminal at $c$} if
$n\in\{\mathrm{Single},\mathrm{Double},\mathrm{Triple},\mathrm{HomeRun},\mathrm{Out}\}$,
or $n=\mathrm{Ball}$ and $b=3$, or $n=\mathrm{Strike}$ and $s=2$; and
\emph{continuing at $c$} otherwise, i.e.\ $n=\mathrm{Ball}$ with $b<3$, or
$n=\mathrm{Strike}$ with $s<2$. Define
\[
W_{s^-}(c,n):=
\begin{cases}
 u\big(\nu(n);s^-\big), & n\text{ terminal at }c,\\[1mm]
 V_{s^-}(b+1,s), & n=\mathrm{Ball}\text{ and }b<3,\\[1mm]
 V_{s^-}(b,s+1), & n=\mathrm{Strike}\text{ and }s<2.
\end{cases}
\]
The boundary condition of this recursion is entirely the top row: it is
fixed once and for all by the terminal utility function of
Definition~\ref{def:utility}, with no reference to $V_{s^-}$. The bottom two
rows are the only place $V_{s^-}$ enters, and only at a count with $b+s$
one larger than at $c$.

Since $b+s$ strictly increases under either non-terminal transition and is
bounded above by $5$ on $\mathcal C$, this is a well-founded backward
recursion on $b+s$, from $5$ down to $0$ -- as for the tree-indexed
recursion of Definition~\ref{def:pitcher-V}, there justified by the
finiteness of the tree (Section~\ref{subsec:history_set}), here by the
strict monotonicity of $b+s$. Its maximum, $b+s=5$, is attained only at the
full count $c=(3,2)$, where $b=3$ and $s=2$ hold simultaneously, so every
$n\in A_N$ is terminal at $(3,2)$: $W_{s^-}((3,2),n)=u(\nu(n);s^-)$ for
every $n$, with no $V_{s^-}$ appearing at all. This is the base case of the
induction; every other count is resolved from values already fixed at
$b+s+1$.
At {\tt count} $c$ define the matrix entries $M_{s^-}(c)$ by:
\begin{equation}
\label{eq:count-state-matrix}
 M_{s^-}(c)[a_P,a_B]
 :=\sum_{n\in A_N}
 \beta_N(n\mid s^-,c,a_P,a_B)\,W_{s^-}(c,n).
\end{equation}
it gives a matrix of size  $|A_P|\times |A_B|$ matrix, expressed in
Batter payoff units.

Although the extensive tree records the Pitcher move before the Batter move,
the Batter information set joins the nodes generated by the different pitch
choices.  Consequently, the strategic interaction at a fixed count is the
zero-sum matrix game $M_{s^-}(c)$.

\begin{proposition}[Count-state equilibrium recursion]
\label{prop:count-state-equilibrium-recursion}
Suppose that Nature's transition kernel depends on the preceding history only
through $(s^-,c,a_P,a_B)$ and that terminal utility depends only on $s^-$ and
the terminal plate-appearance outcome.  Then all histories having the same
pair $(s^-,c)$ have the same continuation value.  This value is determined by
\begin{equation}
\label{eq:count-state-minimax-dp}
 V_{s^-}(c)
 =\min_{\pi_P\in\Delta(A_P)}
   \max_{\pi_B\in\Delta(A_B)}
   \pi_P^{\mathsf T}M_{s^-}(c)\pi_B
 =\max_{\pi_B\in\Delta(A_B)}
   \min_{\pi_P\in\Delta(A_P)}
   \pi_P^{\mathsf T}M_{s^-}(c)\pi_B,
\end{equation}
where the successor values appearing in~\eqref{eq:count-state-matrix} are
computed backward from the terminal boundaries.  Selecting a saddle point of
$M_{s^-}(c)$ at each count produces a minimax behavioral strategy for the full
plate appearance.
\end{proposition}
\begin{proof}
Order the counts backward by the number of additional Balls and Strikes that
can occur before termination.  At a count for which every possible Nature
outcome terminates the plate appearance, the continuation game is exactly the
finite matrix game in~\eqref{eq:count-state-matrix}; hence its value is given
by~\eqref{eq:count-state-minimax-dp}.  Suppose now that the assertion has been
established at every count reachable from $c$ after one non-terminal Ball or
Strike.  Conditional on the current action pair, Nature either terminates the
plate appearance or moves to one of those successor counts.  By the induction
hypothesis, the value at a successor count is independent of the particular
history leading to it.  Thus every history with current state $(s^-,c)$ induces
the same matrix $M_{s^-}(c)$.  The minimax theorem gives its value and a local
saddle point.  Continuing backward to $(0,0)$ proves the value statement.
Concatenating the local saddle-point strategies over the count states gives
behavioral strategies that guarantee this value from every reachable state,
which proves the final assertion.
\end{proof}

Figure~\ref{fig:count-state-dp} displays the state compression established by
Proposition~\ref{prop:count-state-equilibrium-recursion}. For a fixed initial
record $s^-$, each node represents one continuation game, regardless of how
many extensive-form histories lead to that count. Ball and Strike outcomes
move the game through this state space until a terminal boundary is reached,
whereas every in-play outcome terminates the plate appearance immediately.
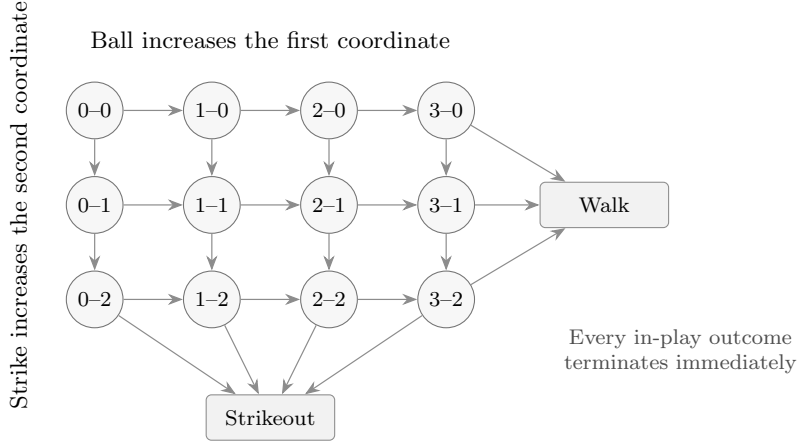
\begin{figure}[H]
  \centering
\begin{tikzpicture}[
  x=1.55cm,y=1.25cm,
  state/.style={circle,draw=black!55,fill=black!3,minimum size=7.5mm,
                inner sep=0pt,font=\footnotesize},
  boundary/.style={rounded corners=2pt,draw=black!45,fill=black!5,
                   minimum width=17mm,minimum height=6mm,font=\footnotesize},
  arr/.style={-{Stealth[length=2mm]},draw=black!45,line width=.45pt},
  ball/.style={font=\scriptsize,above,sloped,text=black!70},
  strike/.style={font=\scriptsize,below,sloped,text=black!70}
]
  \foreach \b in {0,1,2,3}{
    \foreach \s in {0,1,2}{
      \node[state] (c\b\s) at (\b,-\s) {$\b$--$\s$};
    }
  }

  \node[boundary] (walk) at (4.35,-1) {Walk};
  \node[boundary] (so) at (1.5,-3.25) {Strikeout};

  \foreach \b in {0,1,2}{
    \foreach \s in {0,1,2}{
      \pgfmathtruncatemacro{\bp}{\b+1}
      \draw[arr] (c\b\s) -- (c\bp\s);
    }
  }
  \foreach \s in {0,1,2}{
    \draw[arr] (c3\s) -- (walk);
  }

  \foreach \b in {0,1,2,3}{
    \foreach \s in {0,1}{
      \pgfmathtruncatemacro{\sp}{\s+1}
      \draw[arr] (c\b\s) -- (c\b\sp);
    }
    \draw[arr] (c\b2) -- (so);
  }

  \node[font=\small] at (1.5,0.75) {Ball increases the first coordinate};
  \node[font=\small,rotate=90] at (-0.65,-1) {Strike increases the second coordinate};
  \node[font=\footnotesize,align=center,text=black!70] at (5.0,-2.55)
    {Every in-play outcome\\terminates immediately};
\end{tikzpicture}
  \caption{The figure shows the count-state compression obtained under the
standard four-Ball, three-Strike specification. For a fixed initial record
$s^-$, all extensive-form histories reaching the same count share the same
continuation game. Non-terminal Ball and Strike outcomes move play among the
twelve count states, while a fourth Ball, a third Strike, or an in-play outcome
terminates the plate appearance. The minimax values are therefore computed
backward over this count-state space.}
  \label{fig:count-state-dp}
\end{figure}

Therefore, computationally, for each fixed initial record
$s^-$, the standard four-Ball, three-Strike model requires solving one local
zero-sum matrix game at each of the twelve non-terminal counts. With the action
sets specified in Section~\ref{sec:game-theory-model}, these are twelve
$8\times 2$ matrix games.

This conclusion is special to the present state-Markov model and is not a
general replacement for von Stengel's sequence form.  If Nature's distribution
or either player's feasible strategy depends on pitch sequencing, fatigue,
learning, or another feature of the preceding history that is not included in
the state, two histories with the same count need not have the same continuation
game.  The count compression then fails.  The sequence-form formulation remains
the general equilibrium method, while Proposition~\ref{prop:count-state-equilibrium-recursion}
identifies a tractable and empirically relevant subclass.
The idea that a full at-bat can be treated as a state-indexed stochastic
game and solved backward from terminal counts is not unprecedented: Douglas
et al.~\cite{douglas2023atbat} propose exactly this reduction for a baseball
at-bat valued by on-base percentage, and remark that it can in principle be
solved by classical stochastic-game methods, without supplying a proof.
What Proposition~\ref{prop:count-state-equilibrium-recursion} adds is a
proof, under the stated hypotheses, that this reduction is \emph{exact} for
the run-expectancy utility and coincides with the equilibrium of the
general sequence-form program of Appendix~\ref{vonStengelComputationOfEquilibria}
rather than merely a plausible approximation to it; and the action-value and
leverage decomposition of Section~\ref{subsec:count-state-equilibrium-dp}
built on top of it, which we have not found elsewhere in the baseball
game-theory literature.


\subsubsection{Action values and strategic leverage}

The equilibrium mixture reports which actions are used, but not how costly the
unused alternatives would be.  The count-state recursion supplies this missing
information.  Let $(\pi_P^*(c),\pi_B^*(c))$ be a selected saddle point of
$M_{s^-}(c)$.  For a Pitcher action
$a_P$ and a Batter action $a_B$, define the one-step deviation values
\begin{align}
 Q_P(s^-,c,a_P)
   &:=e_{a_P}^{\mathsf T}M_{s^-}(c)\pi_B^*(c),
   \label{eq:pitcher-action-value}\\
 Q_B(s^-,c,a_B)
   &:=\pi_P^*(c)^{\mathsf T}M_{s^-}(c)e_{a_B}.
   \label{eq:batter-action-value}
\end{align}
These quantities force the indicated current action and then value all later
play through the continuation recursion.  Since the matrix is expressed in
Batter payoff units, the action-specific losses are
\begin{align}
 G_P(s^-,c,a_P)&:=Q_P(s^-,c,a_P)-V_{s^-}(c)\geq 0,
 \label{eq:pitcher-action-gap}\\
 G_B(s^-,c,a_B)&:=V_{s^-}(c)-Q_B(s^-,c,a_B)\geq 0.
 \label{eq:batter-action-gap}
\end{align}
The inequalities follow from the saddle-point property.  They vanish for every
action used with positive probability in the selected equilibrium.  Hence the
gaps quantify something not visible from the strategy probabilities alone:
whether an unused action is nearly interchangeable with the equilibrium support
or is strategically costly.

We summarize the sensitivity of a state by
\begin{equation}
\label{eq:conditional-leverage}
 L_P(s^-,c):=\max_{a_P\in A_P}G_P(s^-,c,a_P),
 \qquad
 L_B(s^-,c):=\max_{a_B\in A_B}G_B(s^-,c,a_B).
\end{equation}
This is a \emph{conditional leverage}: it measures how much action selection
matters, conditional on the state $(s^-,c)$ having been reached.  Let
$\rho_{s^-}(c)$ denote the equilibrium probability of reaching count $c$ from
$(s^-,0\text{--}0)$.  The corresponding reach-weighted leverage is
\begin{equation}
\label{eq:reach-weighted-leverage}
 \Lambda_i(s^-,c):=\rho_{s^-}(c)L_i(s^-,c),
 \qquad i\in\{P,B\}.
\end{equation}
The distinction is important.  A poor action can have a dramatic conditional
cost at a count that is almost never reached, while a smaller action gap at a
common count can have a larger effect on the ex ante value of the plate
appearance.  Thus $L_i$ answers ``how consequential is the decision here?'',
whereas $\Lambda_i$ answers ``how much does this decision state matter from the
start of the plate appearance?''

Because $s^-$ enters the terminal utilities, both the equilibrium and the
leverage profile can change with the baseball context even when Nature's
pitch-level kernel is held fixed.  The same count can therefore have different
strategic significance with empty bases, a runner in scoring position, or the
bases loaded.  This is the principal empirical use of the DP output developed
here: it produces a count-by-record map of counterfactual action costs rather
than merely another display of equilibrium mixing probabilities.

\begin{figure}[H]
  \centering
\begin{tikzpicture}[scale=0.98,transform shape,font=\small]
\node[font=\bfseries] at (6.45,1.35)
{Where a poor pitch choice matters: equilibrium reach-weighted leverage};

\begin{scope}[xshift=0cm]
  \node[font=\bfseries\footnotesize,align=center] at (1.8,0.55) {Runner on third, 0 outs};
  \foreach \b/\s/\v in {
        0/0/0.94,
        1/0/0.17,
        2/0/0.05,
        3/0/3.92,
        0/1/0.63,
        1/1/0.36,
        2/1/0.03,
        3/1/22.99,
        0/2/0.38,
        1/2/0.45,
        2/2/0.18,
        3/2/0.13
  }{
    \pgfmathsetmacro{\tint}{4+31*min(\v/23,1)}
    \fill[orange!\tint] (\b,-\s) rectangle ++(0.9,-0.9);
    \draw[gray!60,line width=0.35pt] (\b,-\s) rectangle ++(0.9,-0.9);
    \node[font=\footnotesize] at (\b+0.45,-\s-0.45) {\v};
  }
  \foreach \b in {0,1,2,3}
    \node[font=\footnotesize] at (\b+0.45,-3.08) {\b};
  \foreach \s in {0,1,2}
    \node[font=\footnotesize] at (-0.22,-\s-0.45) {\s};
  \node[font=\footnotesize] at (1.8,-3.45) {Balls};
  \node[font=\footnotesize,rotate=90] at (-0.72,-1.35) {Strikes};
\end{scope}

\begin{scope}[xshift=4.45cm]
  \node[font=\bfseries\footnotesize,align=center] at (1.8,0.55) {Bases loaded, 1 out};
  \foreach \b/\s/\v in {
        0/0/2.79,
        1/0/0.61,
        2/0/0.05,
        3/0/9.24,
        0/1/1.81,
        1/1/1.10,
        2/1/0.21,
        3/1/0.14,
        0/2/1.07,
        1/2/1.27,
        2/2/0.58,
        3/2/0.21
  }{
    \pgfmathsetmacro{\tint}{4+31*min(\v/23,1)}
    \fill[orange!\tint] (\b,-\s) rectangle ++(0.9,-0.9);
    \draw[gray!60,line width=0.35pt] (\b,-\s) rectangle ++(0.9,-0.9);
    \node[font=\footnotesize] at (\b+0.45,-\s-0.45) {\v};
  }
  \foreach \b in {0,1,2,3}
    \node[font=\footnotesize] at (\b+0.45,-3.08) {\b};
  \foreach \s in {0,1,2}
    \node[font=\footnotesize] at (-0.22,-\s-0.45) {\s};
  \node[font=\footnotesize] at (1.8,-3.45) {Balls};
\end{scope}

\begin{scope}[xshift=8.9cm]
  \node[font=\bfseries\footnotesize,align=center] at (1.8,0.55) {Bases empty, 1 out};
  \foreach \b/\s/\v in {
        0/0/1.00,
        1/0/0.32,
        2/0/0.06,
        3/0/2.97,
        0/1/0.54,
        1/1/0.41,
        2/1/0.15,
        3/1/0.28,
        0/2/0.28,
        1/2/0.38,
        2/2/0.24,
        3/2/0.02
  }{
    \pgfmathsetmacro{\tint}{4+31*min(\v/23,1)}
    \fill[orange!\tint] (\b,-\s) rectangle ++(0.9,-0.9);
    \draw[gray!60,line width=0.35pt] (\b,-\s) rectangle ++(0.9,-0.9);
    \node[font=\footnotesize] at (\b+0.45,-\s-0.45) {\v};
  }
  \foreach \b in {0,1,2,3}
    \node[font=\footnotesize] at (\b+0.45,-3.08) {\b};
  \foreach \s in {0,1,2}
    \node[font=\footnotesize] at (-0.22,-\s-0.45) {\s};
  \node[font=\footnotesize] at (1.8,-3.45) {Balls};
\end{scope}

\begin{scope}[xshift=3.85cm,yshift=-4.25cm]
  \node[font=\footnotesize,anchor=east] at (-0.15,0.18)
  {Reach-weighted leverage (milli-runs)};
  \foreach \i in {0,...,20}{
    \pgfmathsetmacro{\p}{\i/20}
    \pgfmathsetmacro{\tint}{4+31*\p}
    \fill[orange!\tint] (0.18*\i,0) rectangle ++(0.18,0.36);
  }
  \draw[gray!70] (0,0) rectangle (3.78,0.36);
  \node[font=\scriptsize,anchor=north] at (0,0) {0};
  \node[font=\scriptsize,anchor=north] at (1.89,0) {11.5};
  \node[font=\scriptsize,anchor=north] at (3.78,0) {23};
\end{scope}
\end{tikzpicture}
  \caption{Preliminary Pitcher reach-weighted leverage
  $1000\,\Lambda_P(s^-,c)$, in thousandths of a run, for three initial
  base--out records. The calculation uses the empirically calibrated three-action Pitcher
specification together with the standard four-Ball, three-Strike count
structure. Darker cells mark
  states at which a poor pitch choice combines a meaningful conditional loss
  with a non-negligible equilibrium reach probability.}
  \label{fig:pitcher-reach-weighted-leverage}
\end{figure}
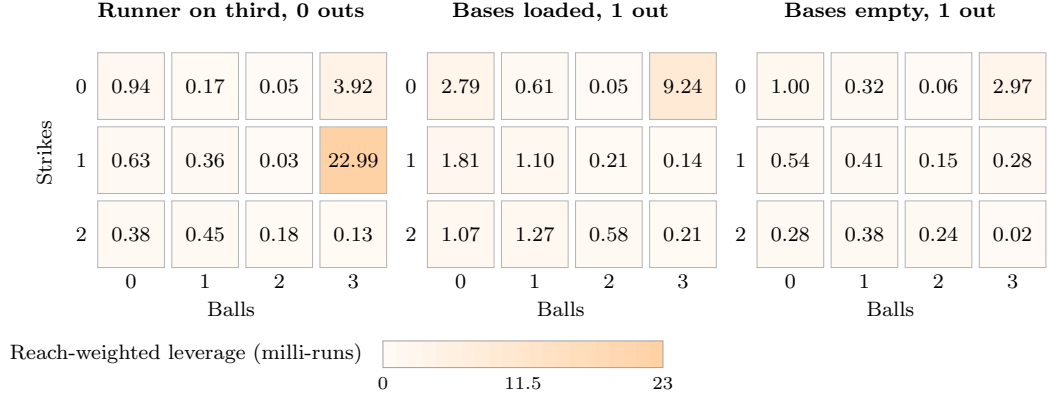

Figure~\ref{fig:pitcher-reach-weighted-leverage} illustrates why both forms of
leverage are useful.  Three-ball counts can carry large conditional penalties
because a Ball is then terminal, but the reach-weighted display prevents a rare
state from dominating the interpretation merely because its conditional loss
is large.  In the present calibration, the largest reach-weighted Pitcher
leverage occurs at the $3$--$1$ count with a runner on third and no outs.  Its
conditional action gap is approximately $0.556$ runs, its equilibrium reach
probability is approximately $0.0414$, and their product is approximately
$0.0230$ runs.  This does not by itself establish empirical optimality or
suboptimality of any observed player.  Rather, it identifies the game contexts
in which observed decisions would be most informative to evaluate.

\begin{remark}[Relation to run expectancy]
The value function $V_{s^-}$ used in this section is a game-theoretically consistent
analogue of the run-expectancy tables used in sabermetrics (e.g.\ RE24,
\cite{tangotiger-rematrix}), indexed by information set --- count, pitch history, and the
acting player's belief --- rather than by base--out state alone, and computed as a best
response rather than fitted from play-by-play averages.
\end{remark}


\section{Results}
\label{sec:results} 
The numerical analysis illustrates how the extensive-form formulation,
the sequence-form linear programs, and the count-state recursion produce
computable equilibrium quantities with distinct strategic interpretations.
We compare equilibrium behavioral strategies across several initial
base--out records $s^-$, while retaining the same underlying game tree and
Nature transition kernel, and we compare the Batter's model-implied swing
probabilities at different counts with empirical frequencies obtained from
Statcast.

The illustrations use two related but distinct computational regimes. First,
Figure~\ref{fig:pitcher-reach-weighted-leverage} reports the leverage
quantities generated by the count-state minimax recursion. That computation
uses the standard four-Ball, three-Strike count structure together with the
empirically calibrated three-action Pitcher specification. It demonstrates
how the state recursion produces equilibrium continuation values, action
values, and reach-weighted measures of the strategic cost of non-optimal
actions.

Second, the computations reported below solve the coupled sequence-form
linear programs on a reduced extensive game. In this game, a walk occurs
after three Balls, a strikeout after two Strikes, and the Pitcher's action
set is
\[
A_P=
\{\text{Fastball\_Bottom},\text{Offspeed\_Bottom},
\text{Offspeed\_Chase}\}.
\]
These restrictions reduce the size of the game tree submitted to the
sequence-form programs while preserving the multi-stage
imperfect-information structure of the plate appearance. The resulting
equilibrium realization plans are converted into behavioral strategies and
compared across three initial records along the same fixed path through the
game tree. The two computational regimes therefore serve complementary
purposes and are not intended for direct numerical comparison: the first
illustrates the full count-state recursion and its leverage measures, while
the second illustrates exact sequence-form equilibrium computation on an
explicit extensive game.

\begin{figure}[h!]
    \centering
    \includegraphics[width=1.0\linewidth]{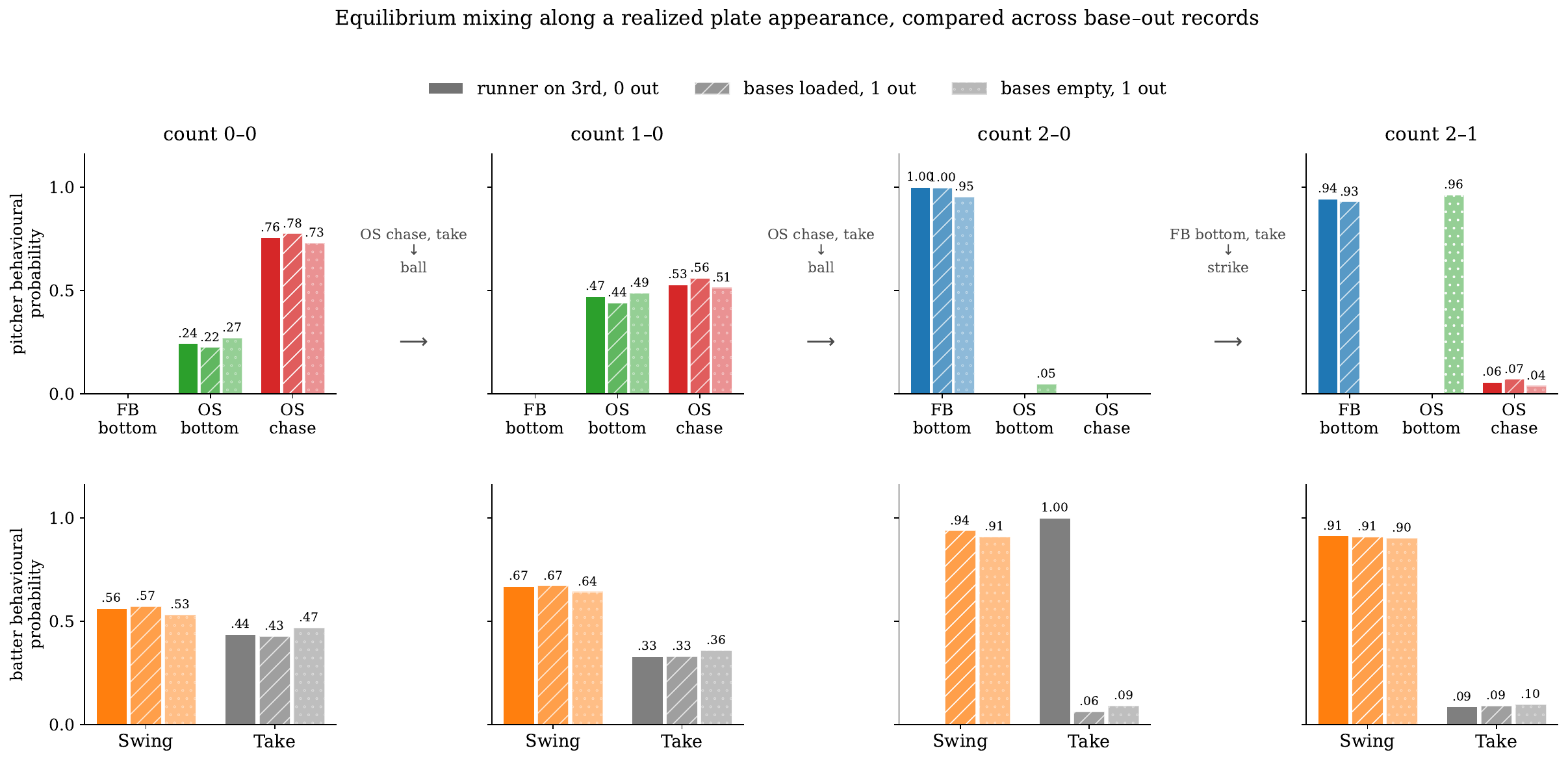}
    \caption{Note the arrow in between each count specifies the edges taken in this fixed path in the game tree. For example, between count 1-0 and 2-0, even though both agents have distributions over their action, we follow the deterministic path in the tree defined by the edges representing agent actions of \textit{Offspeed\_Chase} for the pitcher followed by batter \textit{Take} and nature result of a \textit{Ball}. The same logic is applied for actions taken at all the pitch cycles.}
    \label{fig:pathwiseEquilibriumMixedStrategies}
\end{figure}

\medskip

\textbf{Observations: } We have the strategies for both agents are very similar between all records except when the counts get closer to the boundary condition. Here we see for $s^-$ corresponding to records with runners on base, when the count gets close to a walk, the pitcher favours pitches that are in the zone and the batter. On the other hand, when there is a runner on third, the batter favors taking a pitch. Finally, we see when teh count is 2-1, the equilibrium strategy favors an aggressive approach looking to avoid a strikeout.

Below is a figure comparing the swing-take equilibrium strategy for the batter at different counts with that of the empirical frequency computed by Statcast.

\begin{figure}[H]
\centering
\includegraphics[width=0.9\linewidth]{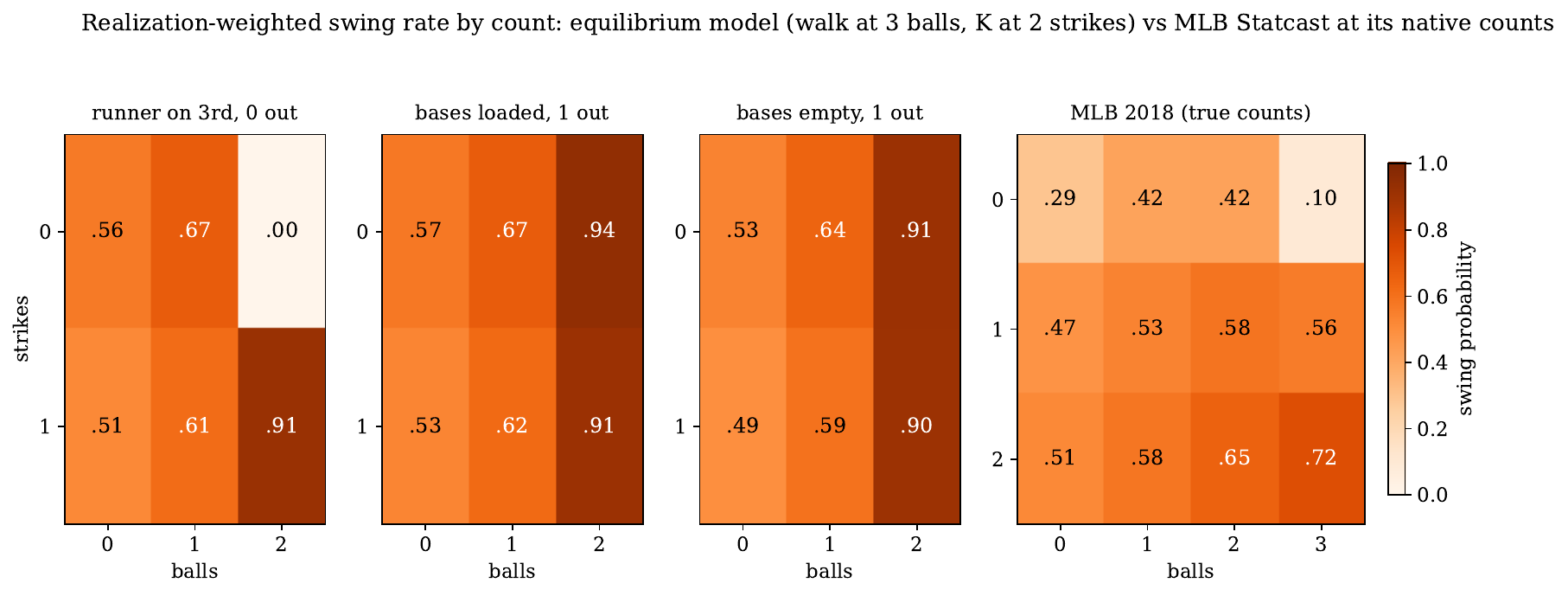}
\caption[Equilibrium swing rates by count]
{Equilibrium swing rates by count, compared with MLB Statcast data.}
\label{fig:historicalSwingRatesConditionalOnCount}
\begin{minipage}{0.9\linewidth}
\small
Let $(x,y)$ denote the equilibrium realization plans obtained from the joint
LPs (\ref{def:joint-lps}) (in Appendix \ref{vonStengelComputationOfEquilibria}), and let $\beta_B$ be the
behavioural strategy induced by $x$ via Proposition~4.2. Every node of a
batter information set $u \in \mathcal{U}_B$ shares the same history of
Nature outcomes, so $\mathrm{count}(u)=(b,s)$ is well defined
(Definition~1.1), as is the batter sequence $\sigma_B(u)$
(Proposition~4.1). The first three panels show, for each initial record
$s^-$, the equilibrium probability that the Batter swings at the next pitch
conditional on the plate appearance reaching count $(b,s)$:
\[
S(b,s)
=
\frac{
    \displaystyle
    \sum_{u:\,\mathrm{count}(u)=(b,s)}
    x\!\left(\sigma_B(u)\right)
    P_y(A_u)\,
    \beta_B\!\left(\textsc{Swing}\mid u\right)
}{
    \displaystyle
    \sum_{u:\,\mathrm{count}(u)=(b,s)}
    x\!\left(\sigma_B(u)\right)
    P_y(A_u)
}.
\]
Here $P_y(A_u)$ is the reach weight of Definition~5.6, so that
$x(\sigma_B(u))\,P_y(A_u)$ is the equilibrium probability of reaching $u$;
information sets unreached in equilibrium receive zero weight.

In the reported equilibria, all reached information sets at a given count
play an identical mix, so each cell equals that common value. The model
awards a walk at three balls and a strikeout at two strikes
(Definition~1.2), so its grid spans $b \in \{0,1,2\}$,
$s \in \{0,1\}$. The rightmost panel shows league-wide MLB swing
frequencies (swings divided by pitches thrown at each count; 2018 MLB
Statcast pitch-level data~\cite{mlb-statcast}).
\end{minipage}
\end{figure}

\section{Conclusion}
\label{sec:conclusion}

We have formulated a complete baseball plate appearance as a finite
two-person, zero-sum, extensive-form game with chance moves and imperfect
information. The evolving count connects successive pitch cycles, while the
Batter's information sets represent the fact that the current pitch choice is
not observed before the Swing-or-Take decision. Terminal outcomes are valued
through a run-expectancy utility that incorporates both the immediate runs
scored and the continuation value of the half-inning. This formulation places
the Pitcher--Batter confrontation within a single dynamic game rather than a
collection of unrelated pitch-level matrix games.

The main computational framework is the sequence-form representation. It
replaces optimization over the exponentially large space of complete
pure-strategy plans by coupled linear programs whose dimensions are linear in
the extensive game tree. The resulting realization plans yield exact minimax
behavioral strategies while respecting the information and perfect-recall
constraints of the underlying game. Together with the Statcast-estimated
transition kernel for Nature, this provides a complete computational pipeline
from the specification of the dynamic interaction to equilibrium strategies
and their baseball interpretation.

The paper also identifies two distinct roles for dynamic programming. For a
general finite two-person zero-sum extensive-form game with perfect recall,
once the opponent's realization plan is fixed, the dual variables of the
sequence-form best-response program decompose into reach weights and
conditional continuation values. Under the additional state-Markov structure
of the present baseball model, the equilibrium problem itself admits a sharper
reduction: the full minimax value and equilibrium strategies can be computed
backward over the twelve non-terminal counts by solving one local zero-sum
matrix game at each count. These two computations are complementary. The
sequence form applies to the full history-dependent extensive game, whereas
the count-state recursion exploits the more restrictive setting in which the
pair consisting of the initial base--out record and the current count provides
a sufficient description of the continuation game.

The count-state computation also produces equilibrium action values and
conditional and reach-weighted leverage measures. These quantities distinguish
the cost of a non-optimal action conditional on reaching a count from its
overall strategic importance after the probability of reaching that count is
taken into account. The numerical results therefore illustrate not only how
equilibrium strategies change with the initial base--out record, but also how
the computational framework can identify the stages of the plate appearance
at which strategic choices have the greatest consequences.

The present specification is deliberately parsimonious, but its main
components---the players' action sets, the Batter's information, Nature's
transition probabilities, and the terminal utility---can be refined without
changing the underlying structure of the model. Because the game is formulated
in sequence form, such extensions remain connected to a practical method for
computing equilibria, one that has been applied elsewhere to substantially
larger extensive-form games, although we do not explore that scale here. We
have also provided the definitions and computational details needed to
reproduce, modify, and extend the numerical implementation. The software and
accompanying documentation are available at
\url{https://github.com/ekohn63/NashEquilibrium/tree/strat_infoset}.

\section*{Acknowledgements}

The authors thank Aidan Schneider from the Department of Mathematics and Statistics at the University of Victoria for useful discussions at the beginning of the project.

\appendix

\section{Baseball Terminology and Conventions}
\label{app:baseball-background}

This appendix collects the basic baseball terminology used in
Section~\ref{sec:plate-appearance-model}. A run is scored when a player advances through all four bases and reaches
home plate. A baseball game is normally played over nine innings. If the score is tied after nine innings, additional innings are played until an inning ends with one team ahead. The team scoring the greater number of runs wins the game. Each inning is divided into two half-innings. During a half-inning, one team is on offense and sends players to bat, while the other team is on defense and attempts to record outs. The offensive team seeks to advance runners around the bases and score runs. A half-inning ends when the defensive team records
three outs, after which the teams exchange offensive and defensive roles.

A batter who reaches first, second, or third base becomes a runner. Runners may
advance when the batter puts the ball in play, although advancement may also
occur through other events such as a walk. Baseball permits several ways of
recording an out. For the purposes of the record-transition map introduced in
Definition~\ref{def:state}, an out in play and a strikeout both increase the
number of outs by one, although they arise from different terminal histories
in the extensive game.

A plate appearance is one completed turn by a Batter. In our model, its
terminal outcomes are
\[
\{\text{Single},\text{Double},\text{Triple},\text{Home Run},
  \text{Out},\text{Strikeout},\text{Walk}\}.
\]
A single, double, or triple is classified according to whether the Batter
reaches first, second, or third base, respectively. On a home run, the Batter
and all runners already on base score. A walk places the Batter on first base
and forces other runners to advance when necessary.

A half-inning therefore consists of a sequence of plate appearances ending
when three outs have been recorded. Each plate appearance may score runs and
changes the base--out record. In the mathematical model, a terminal history
$z\in Z$ records the sequence of actions leading to the terminal outcome of
one plate appearance.


\section{Linear Programming Formulation}   \label{linearProgramFormulation}

\subsection{Zero-Sum Equilibrium}

\begin{lemma} 
\label{bestResponse}[Mixed Best Response Support~\footnote{This is equivalent to the "Best Response Condition" in Proposition 3.1 of the Algorithmic Game Theory book \cite{agt_nisan}, which implies that a player's mixed best response support consists entirely of pure strategies that individually achieve the optimal value.}]
Given a mixed strategy $\pi_B$ there is always a pure strategy $s_P \in \Sigma_P$ that is as good as a mixed strategy $\pi_P$. That is,
\[
\min_{\pi_P} \pi_B^T\,\,U\text{ }\pi_P = \min_{s_P \in \Sigma_P}\mathbb{E}_{X \sim \pi_{B}}[U(X,s_P)] \,\, = \min_{s_P \in \Sigma_P} \sum_{s_B \in \Sigma_B}U(s_B,s_P) \cdot \pi_B(s_B).
\] 

Similarly, given a mixed strategy $\pi_P$ there is a pure strategy $s_B$ that is as good as a mixed strategy $\pi_B$. That is: 
\[
\max_{\pi_B} \, \pi_B^T \,\, U \,\,\pi_P = \max_{s_B \in \Sigma_B} \sum_{s_P \in \Sigma_P}U(s_B,s_P)\cdot \pi_P(s_P).
\]
\end{lemma}

\subsubsection{Linear Program Formulation for Computing the Equilibrium}

Under the zero-sum framework, the batter selects a mixed strategy $\pi_B \in \Delta(\Sigma_B)$ to maximize their expected payoff against an adversarial pitcher who simultaneously chooses a mixed strategy $\pi_P \in \Delta(\Sigma_P)$ to minimize it. Symmetrically, for any fixed mixed strategy chosen by the batter, the pitcher seeks an optimal response to minimize this same expected value:
\[
\min_{\pi_P \in \Delta(\Sigma_P)} \sum_{s_B \in \Sigma_B} \sum_{s_P \in \Sigma_P} \pi_B(s_B) \, \pi_P(s_P) \, U(s_B, s_P)
\]

For a fixed mixed strategy $\pi_B$ of the batter, the pitcher seeks an optimal adversarial counter-strategy to minimize the expected payoff. Using Lemma \ref{bestResponse} above, this minimization over the compact strategy simplex reduces to an optimization over the pitcher's pure strategy set $\Sigma_P$. This optimal response defines the batter's guaranteed security level, denoted by the scalar $\alpha \in \mathbb{R}$:
\[
\alpha := \min_{\pi_P} \pi_{B}^\top U \pi_{P} = \min_{s_P \in \Sigma_P} \sum_{s_B \in \Sigma_B} U(s_B, s_P) \, \pi_B(s_B)
\]

Anticipating this optimal adversarial minimization by the pitcher, the batter selects a mixed strategy $\pi_B \in \Delta(\Sigma_B)$ to maximize their guaranteed security level. This yields the standard maximin optimization problem:
\[
\max_{\pi_B \in \Delta(\Sigma_B)} \min_{\pi_P \in \Delta(\Sigma_P)} \pi_{B}^\top U \pi_{P} = \max_{\pi_B \in \Delta(\Sigma_B)} \min_{s_P \in \Sigma_P}  \sum_{s_B \in \Sigma_B} U(s_B, s_P) \, \pi_B(s_B)
\]

By treating the inner minimization over the finite set $\Sigma_P$ as a system of simultaneous bounding constraints, the batter's objective can be cast as the following primal linear programming problem:
\begin{equation}
\label{lp:primal-strategic}
\begin{aligned}
\max_{\alpha,\;\pi_B} \quad & \alpha \\
\text{s.t.}\quad 
& \alpha - \sum_{s_B \in \Sigma_B} U(s_B, s_P) \, \pi_B(s_B) \le 0
&& \forall\, s_P \in \Sigma_P,\\[4pt]
& \sum_{s_B \in \Sigma_B} \pi_B(s_B) = 1,\\
& \pi_B(s_B) \ge 0 
&& \forall\, s_B \in \Sigma_B,\\[2pt]
& \alpha \in \mathbb{R}.
\end{aligned}
\end{equation}

Symmetrically, evaluating the game from the pitcher's perspective, for any fixed mixed strategy $\pi_P \in \Delta(\Sigma_P)$, the batter will choose an optimal counter-strategy to maximize the expected payoff. Invoking the second component of the preceding lemma, this optimization similarly reduces to a search over the batter's pure strategy space $\Sigma_B$. The pitcher's guaranteed upper bound on expected exposure, denoted by the scalar $\beta \in \mathbb{R}$, is explicitly defined as:
\[
\beta := \max_{\pi_B \in \Delta(\Sigma_B)} \pi_{B}^\top U \pi_P = \max_{s_B \in \Sigma_B} \sum_{s_P \in \Sigma_P} U(s_B, s_P) \, \pi_P(s_P)
\]

The pitcher seeks a mixed strategy to minimize this maximum guaranteed payoff. This minimax objective is formally expressed as the following dual linear programming problem:
\begin{equation}
\label{lp:dual-strategic}
\begin{aligned}
\min_{\beta,\;\pi_P} \quad & \beta \\
\text{s.t.}\quad 
& \beta - \sum_{s_P \in \Sigma_P} U(s_B, s_P) \, \pi_P(s_P) \ge 0 
&& \forall\, s_B \in \Sigma_B,\\[4pt]
& \sum_{s_P \in \Sigma_P} \pi_P(s_P) = 1,\\
& \pi_P(s_P) \ge 0 
&& \forall\, s_P \in \Sigma_P,\\[2pt]
& \beta \in \mathbb{R}.
\end{aligned}
\end{equation}

\begin{theorem}
\label{thm:pure-strat-lp}
There exist optimal mixed strategies $\pi_B^\star \in \Delta(\Sigma_B)$ and $\pi_P^\star \in \Delta(\Sigma_P)$ for the strategic-form game such that:
\[
\max_{\pi_B \in \Delta(\Sigma_B)} \min_{\pi_P \in \Delta(\Sigma_P)} \pi_B^{\top} U \pi_P \;=\; \min_{\pi_P \in \Delta(\Sigma_P)} \max_{\pi_B \in \Delta(\Sigma_B)} \pi_B^{\top} U \pi_P \;=\; (\pi_B^\star)^{\top} U \pi_P^\star
\]
The strategy profile $(\pi_B^\star, \pi_P^\star)$ constitutes a saddle-point minimax solution, corresponding to a Nash equilibrium of the zero-sum game.
\end{theorem}

\begin{proof}
Let \eqref{lp:primal-strategic} serve as the primal linear programming problem representing the batter's maximin optimization model. We apply standard algebraic rules of linear programming duality to construct its formal dual. 

To do so, let $\pi_P(s_P) \ge 0$ be the dual multipliers assigned to the system of structural security constraints $\alpha - \sum_{s_B \in \Sigma_B} U(s_B, s_P) \, \pi_B(s_B) \le 0$ for each $s_P \in \Sigma_P$. Similarly, let $\beta \in \mathbb{R}$ be the free dual variable corresponding to the primal strategy normalization constraint $\sum_{s_B \in \Sigma_B} \pi_B(s_B) = 1$. Collecting coefficients yields the dual objective of minimizing $\beta$, subject to the dual constraints $\beta - \sum_{s_P \in \Sigma_P} U(s_B, s_P) \, \pi_P(s_P) \ge 0$ for each $s_B \in \Sigma_B$, alongside the probability simplex conditions $\sum_{s_P \in \Sigma_P} \pi_P(s_P) = 1$ and $\pi_P(s_P) \ge 0$. By structural inspection, this exact transformation yields the pitcher's optimization model formulated in \eqref{lp:dual-strategic}.

Both linear programs are inherently feasible; assigning a uniform probability distribution over the finite sets $\Sigma_B$ and $\Sigma_P$ satisfies the structural restrictions of both spaces. Furthermore, because the pure strategy sets are finite, the objectives are bounded over the compact simplices. By the Strong Duality Theorem of Linear Programming, the existence of a feasible primal-dual pair guarantees that their respective optimal objective values exist and must coincide, such that $\alpha^\star = \beta^\star$. This exact algebraic identity fulfills Von Neumann's minimax condition, establishing the existence of the optimal mixed strategy pair $(\pi_B^\star, \pi_P^\star)$ and completing the proof.
\end{proof}

\section{von Stengel's Computation of Equilibria} \label{vonStengelComputationOfEquilibria}

\subsection{Realization Plans and Behavioral Strategies}
\label{app:realization-behavioral}
\begin{proposition} \label{prop:realization-plan-induced}
 For any valid realization plan $r_i$, there exists a behavioral strategy profile $\beta_i$ that induces $r_i$.
\end{proposition}
\begin{proof}
We construct a candidate behavioral strategy profile $\beta_i$ and verify that it induces $r_i$. For each information set $v \in \mathcal{U}_i$ and action $a \in A(v)$, define:
\[
\begin{cases}
    \beta_i(a \mid v) = \frac{r_i(\sigma_v a)}{r_i(\sigma_v)}, & \text{if } r_i(\sigma_v) > 0, \\
    \beta_i(\cdot \mid v) \text{ is an arbitrary probability distribution over } A(v), & \text{if } r_i(\sigma_v) = 0.
\end{cases}
\]
Since $r_i$ is a realization plan, it satisfies the conservation
constraint~\eqref{realizationPlanConstraints2}. Thus, when $r_i(\sigma_v) > 0$, summing over all available actions yields $\sum_{a \in A(v)} \beta_i(a \mid v) = \frac{\sum_{a \in A(v)} r_i(\sigma_v a)}{r_i(\sigma_v)} = \frac{r_i(\sigma_v)}{r_i(\sigma_v)} = 1$, confirming $\beta_i(\cdot \mid v)$ is a valid probability distribution. The proof that this constructed strategy profile induces $r_i$ then follows directly by induction on the length of the sequence paths.
\end{proof}

\[
\mathcal B_i
:=
\prod_{U_i^j\in\mathcal U_i}
\Delta\!\bigl(A(U_i^j)\bigr),
\]
\[
\mathcal R_i
:=
\left\{
r_i:S_i\longrightarrow\mathbb R_{+}:
\begin{array}{l}
r_i(\varnothing)=1,\\[0.15em]
r_i(\sigma_v)=
\displaystyle\sum_{a\in A(v)}r_i(\sigma_v a),
\quad v\in\mathcal U_i
\end{array}
\right\}.
\]

\begin{figure}[H]
\centering

{\large
\begin{tikzcd}[
    column sep=5.5em,
    row sep=3.8em
]
\Sigma_i
    \arrow[r, hook, "\iota_i"]
    \arrow[d, hook, "\delta_i"']
&
\mathcal B_i
    \arrow[d, two heads, "\rho_i"]
\\
\Delta(\Sigma_i)
    \arrow[r, two heads, "\Phi_i"']
&
\mathcal R_i
\end{tikzcd}
}

\vspace{0.8em}

\[
\begin{aligned}
\delta_i(s_i)
    &=\delta_{s_i},
&
\iota_i(s_i)
    &=\text{the deterministic behavioral strategy induced by }s_i,
\\[0.4em]
\rho_i(\beta_i)(\sigma)
    &=\prod_{c\in\sigma}\beta_i(c),
&
\Phi_i(\pi_i)
    &=\sum_{s_i\in\Sigma_i}
      \pi_i(s_i)\,
      \rho_i\!\bigl(\iota_i(s_i)\bigr).
\end{aligned}
\]

\[
\Phi_i\circ\delta_i
=
\rho_i\circ\iota_i,
\qquad
\mathcal R_i
=
\rho_i(\mathcal B_i)
=
\operatorname{conv}
\left\{
\rho_i\!\bigl(\iota_i(s_i)\bigr):
s_i\in\Sigma_i
\right\}.
\]

\caption{Relations among pure, mixed, and behavioral strategies and
realization plans for player \(i\). The maps \(\rho_i\) and \(\Phi_i\)
are generally many-to-one. Under perfect recall, mixed and behavioral
strategies have the same image in the realization-plan polytope
\(\mathcal R_i\).}
\label{fig:strategy-realization-diagram}
\end{figure}

\[
\begin{aligned}
\iota_i &: \text{pure strategy}
    &&\longmapsto \text{deterministic behavioral strategy},\\
\delta_i &: \text{pure strategy}
    &&\longmapsto \text{Dirac mixed strategy},\\
\rho_i &: \text{behavioral strategy}
    &&\longmapsto \text{induced realization plan},\\
\Phi_i &: \text{mixed strategy}
    &&\longmapsto \text{convex combination of pure realization plans}.
\end{aligned}
\]

\subsection{Sequence-Form Linear Programs}
\label{app:sequence-form-linear-programs}

Given realization plans $r_1$ and $r_2$, we will write them as column vectors $x$ and $y$ for the batter and pitcher respectively. Note $x$ has $|S_B|$ components and $y$ has $|S_P|$ components.

We define matrices $E$ and $F$ to capture that vectors $x$ and $y$ constitute valid realization plans as defined in Definition \ref{def:realization-plan} via the system:
\[
Ex = e \qquad \qquad Fy = f
\]
where $e \in \mathbb{R}^{1 + |\mathcal{U}_B|}$ and $f \in \mathbb{R}^{1 + |\mathcal{U}_P|}$ are conservation vectors whose first entries (corresponding to the empty sequence $\varnothing$) equal $1$, and all subsequent entries are $0$.

The payoff matrix $A$ for the batter is constructed with entries indexed by sequences:
\[
A_{\sigma_P, \sigma_B} = \sum_{\sigma_N \in S_N} g(\sigma_N, \sigma_B, \sigma_P) \cdot r_N(\sigma_N).
\]
Assuming a zero-sum relationship, the Pitcher's matrix satisfies $B = -A$. The expected payoffs can then be expressed globally using matrix bilinear forms: $x^T A y$ for the Batter and $x^T B y$ for the Pitcher.

\subsection{Best Response Formulations}
To compute a Nash equilibrium for this sequence-form game, we first characterize how an individual player optimally counters a fixed strategy of their opponent. Given a fixed realization plan $x$ for the batter, the pitcher solves a linear program to maximize expected payoff. Conversely, taking the algebraic dual of this program yields the security value bounds enforced by the network structure:

\noindent
\begin{minipage}[t]{0.48\textwidth}
\centering
\textbf{Pitcher Primal Optimization:}
\[
\begin{aligned}
\max_{y} \quad & (x^T B) y \\
\text{s.t.} \quad & F y = f, \\
& y \ge 0.
\end{aligned}
\]
\end{minipage}%
\hfill
\begin{minipage}[t]{0.48\textwidth}
\centering
\textbf{Pitcher Dual Optimization:}
\[
\begin{aligned}
\min_{q} \quad & f^T q \\
\text{s.t.} \quad & F^T q \ge B^T x, \\
& q \ \text{free}.
\end{aligned}
\]
\end{minipage}

\vspace{0.5cm}
\noindent
Symmetrically, if the pitcher's realization plan $y$ is fixed, the batter maximizes expected utility according to a corresponding primal-dual pair:

\noindent
\begin{minipage}[t]{0.48\textwidth}
\centering
\textbf{Batter Primal Optimization:}
\[
\begin{aligned}
\max_{x} \quad & x^T (A y) \\
\text{s.t.} \quad & E x = e, \\
& x \ge 0.
\end{aligned}
\]
\end{minipage}%
\hfill
\begin{minipage}[t]{0.48\textwidth}
\centering
\textbf{Batter Dual Optimization:}
\[
\begin{aligned}
\min_{p} \quad & e^T p \\
\text{s.t.} \quad & E^T p \ge A y, \\
& p \ \text{free}.
\end{aligned}
\]
\end{minipage}
\vspace{0.5cm}

\subsection{Joint Linear Program for Zero-Sum Games}
By leveraging linear programming duality, we can combine the optimizations of both players into a single, unified problem. Substituting the constraints from the batter's dual optimization problem into the pitcher's strategy space yields a single primal-dual system that solves for the game's saddle-point equilibrium simultaneously:

\begin{equation}
\label{def:joint-lps}
\begin{array}{c@{\qquad\qquad}c}
\text{\bfseries Joint Primal LP}
&
\text{\bfseries Joint Dual LP}
\\[1.5ex]
\begin{aligned}
\min_{y,p}\quad & e^T p \\
\text{s.t.}\quad & E^T p-Ay\geq 0, \\
& Fy=f, \\
& y\geq 0
\end{aligned}
&
\begin{aligned}
\max_{x,q}\quad & f^T q \\
\text{s.t.}\quad & F^T q-B^T x\leq 0, \\
& Ex=e, \\
& x\geq 0
\end{aligned}
\end{array}
\end{equation}

\begin{theorem}
The Nash equilibrium strategies and safety values of the zero-sum extensive-form game are given explicitly by the solution vector to the joint primal and dual programs in \eqref{def:joint-lps}.
\end{theorem}
\begin{proof}
The result follows directly from the classical Von Neumann minimax theorem applied to extensive-form games with perfect recall, mapped via strong linear programming duality. For a complete treatment of the structural equivalence between extensive-form saddle points and sequence-form linear program optimizations, see the foundational proofs by Koller et al.~\cite{koller1994fast} and von Stengel~\cite{vonstengel_1996}.
\end{proof}

\end{document}